\documentclass[a4paper, 12pt]{amsart}

\usepackage{verbatim}
\usepackage{amssymb}
\usepackage{amsbsy}
\usepackage{amscd}
\usepackage{amsmath}
\usepackage{amsthm}
\usepackage[mathscr]{eucal}
\usepackage{mathrsfs}  
\usepackage{bm}  
\usepackage[dvipdfmx]{hyperref}


\newtheorem{theorem}{Theorem}\numberwithin{theorem}{section}
\newtheorem{defn}[theorem]{Definition}
\newtheorem{ex}[theorem]{Example}
\newtheorem{prop}[theorem]{Proposition}
\newtheorem{lem}[theorem]{Lemma} 
\newtheorem{cor}[theorem]{Corollary}
\newtheorem{rem}[theorem]{Remark}
\newtheorem{conj}[theorem]{Conjecture}

\numberwithin{equation}{section}

\def\rank{\operatorname{rank}}

\def\L{\mathbb L}
\def\b{\bm{b}} 
\def\x{\bm{x}} 
\def\y{\bm{y}}

\def\ox{\overline{x}}  
\def\ou{\overline{u}}

\newcommand{\Q}{\mathbb{Q}}
\newcommand{\Z}{\mathbb{Z}}

\newcommand{\F}{\mathcal{F}}

\newcommand{\ol}{\overline}

\newcommand{\twoheadlongrightarrow}{\relbar\joinrel\twoheadrightarrow}

\title
{
   The universal factorial Hall--Littlewood $P$- and $Q$-functions
}

\author{Masaki Nakagawa   and Hiroshi Naruse}

\dedicatory{Dedicated to the memory of  Professor Piotr Pragacz}

\address{Graduate School of  Education \endgraf
                       Okayama University \endgraf
                       Okayama  700-8530 \\ Japan 
}
\email{nakagawa@okayama-u.ac.jp}

\address{Graduate School of Education \endgraf 
                        University of Yamanashi   \endgraf
                       Kofu  400-8510 \\ Japan
}

\email{hnaruse@yamanashi.ac.jp}



\begin{document}

\begin{abstract}
     In this paper, we introduce   {\it factorial}   analogues of   the ordinary 
    Hall--Littlewood $P$- and $Q$-polynomials,  which  we call 
   the {\it factorial Hall--Littlewood $P$- and $Q$-polynomials}. 
    Using the {\it universal} formal group law,  we further generalize
    these polynomials  to the {\it universal factorial Hall--Littlewood 
   $P$- and $Q$-functions}.   
    We show that these functions satisfy  the {\it vanishing property} 
    which   the ordinary factorial Schur $S$-, $P$-, and $Q$-polynomials 
    have.  By the vanishing property, we derive the Pieri-type formula
    and a certain generalization of the classical hook formula.      
      We then  characterize  our functions  in terms of Gysin maps from flag bundles 
    in the complex cobordism theory. 
  Using this characterization and Gysin formulas for flag bundles, 
      we can obtain  generating functions for  the  universal   factorial 
     Hall--Littlewood $P$- and $Q$-functions.      
        Using our generating functions, we can show that 
       our factorial Hall--Littlewood $P$- and $Q$-polynomials 
      have a certain {\it cancellation property}. 
      Further applications such as   Pfaffian formulas for 
          $K$-theoretic  factorial $Q$-polynomials  are also  given.  
\end{abstract}

\maketitle


\section{Introduction}     \label{sec:Introduction}  
Let $\x_{n} = (x_{1}, \ldots, x_{n})$ and $t$ be independent indeterminates over 
$\Z$, and $\lambda = (\lambda_{1}, \ldots, \lambda_{n})$   
 a partition of length $\leq n$.    Then the ordinary {\it Hall--Littlewood} $P$- and 
$Q$-{\it polynomials}, denoted  by $P_{\lambda}(\x_{n} ;t)$ and $Q_{\lambda}(\x_{n}; t)$ 
respectively,  are  symmetric polynomials
with coefficients in $\Z [t]$.   
   When $t = 0$,  both the polynomials $P_{\lambda}(\x_{n}; t)$ and $Q_{\lambda}(\x_{n}; t)$ 
 reduce to the 
ordinary Schur  ($S$-)  polynomial $s_{\lambda}(\x_{n})$, and when $t = -1$,  
to the ordinary 
Schur $P$-polynomial  $P_{\lambda}(\x_{n})$ and $Q$-polynomial $Q_{\lambda}(\x_{n})$
respectively. 
   Thus the polynomials $P_{\lambda}(\x_{n}; t)$, $Q_{\lambda}(\x_{n}; t)$ 
serve to interpolate between the Schur  polynomials and the Schur $P$- and 
$Q$-polynomials, 
 and play a crucial role in the  symmetric function theory, representation theory, and
combinatorics.   In the context of {\it Schubert calculus}, 
it is well-known  that  the ordinary Schur  $S$-, $P$-, and $Q$-polynomials 
appear as  the Schubert classes in the ordinary cohomology rings 
of  the various Grassmannians (Fulton \cite[\S 9.4]{Fulton1997}, 
 Pragacz \cite[\S 6]{Pragacz1991}).    Moreover, their {\it factorial} analogues, 
namely, the factorial Schur $S$-, $P$-, and $Q$-polynomials play an analogous 
role in {\it equivariant} Schubert calculus (Knutson--Tao \cite{Knutson-Tao2003},  
Ikeda \cite{Ikeda2007},  Ikeda--Naruse \cite{Ikeda-Naruse2009}).  
Although the Hall--Littlewood polynomials have  no 
obvious   geometric meaning at present,  they deserve to be 
considered in terms of geometry as well.  
In fact,    in \cite{Totaro2003}, Totaro 
considered the {\it coinvariant ring}  $F(e, n)$  of the complex reflection group 
$G(e, 1, n)  =  \Z/e\Z \wr S_{n}$ (the wreath product) for $e \geq 2$, and 
suggested to think of the ring $F(e, n)$ as the cohomology of  a certain ``flag manifold''. 
Then he considered a  subring $C(e, n)$ of $F(e, n)$,  and described a basis for the ring  
$C(e, n)$ given by the Hall--Littlewood $Q$-polynomials.  
For $e = 2$, the inclusion $C(2, n) \subset F(2, n)$ is the inclusion of the 
cohomology of the Lagrangian Grassmannian in that of the isotropic flag manifold 
of  the symplectic group,   and   Totaro's result is interpreted as a   generalization of the classical 
result in Schubert calculus for Lagrangian Grassmannians 
(J\'{o}zefiak \cite{Jozefiak1991},  Pragacz \cite[\S 6]{Pragacz1991}). 
It is natural to consider a generalization of the above theory to the {\it double coinvariant rings} 
(or {\it equivariant coinvariant rings}) of complex reflection groups 
(cf.  recent work of McDaniel \cite{McDaniel2016}). 
From a geometric or topological  point of view,   one expects that 
these rings  would be  related 
to {\it torus-equivariant cohomology} of  certain ``flag manifolds'', 
and factorial version of the Hall--Littlewood polynomials 
would   play a crucial role.   Moreover we notice that 
all the  results stated above  
are formulated in the ordinary cohomology theory $H^{*}(-)$. 
 In topology, it is classical that 
a {\it complex-oriented generalized cohomology theory} $h^{*}(-)$ gives rise to 
a {\it formal group law} $F^{h}(u, v)$ over the coefficient ring $h^{*} := h^{*}(\mathrm{pt})$, 
where $\mathrm{pt}$ is a single point.  
Three typical examples are the ordinary cohomology theory $H^{*}(-)$, 
the (topological) complex $K$-theory $K^{*}(-)$, and the complex cobordism 
theory $MU^{*}(-)$,  which correspond to 
the {\it additive}  formal group law $F_{a} (u, v) = u + v$, 
 the {\it multiplicative} formal group law 
$F_{m}(u, v) = u \oplus v = u + v - \beta uv$, and 
 the {\it universal} formal group law 
$F_{\L}(u, v) = u +_{\L} v$, respectively.  
By the classical result of  Quillen \cite[Proposition 1.10]{Quillen1971},
the complex cobordism theory is {\it universal} among all complex-oriented 
generalized cohomology theories.   
 Therefore it is  also quite natural to ask  whether one can generalize 
the above results formulated in the ordinary cohomology theory 
to the complex cobordism theory.

Motivated by these facts and the above preceding results, 
in this paper, we introduce   {\it factorial}  and {\it universal} analogues of 
 the ordinary  Hall--Littlewood $P$- and $Q$-polynomials,   which we call 
   the {\it universal  factorial Hall--Littlewood $P$- and $Q$-functions}
(For notation, see \S \ref{subsec:UFGL}):  
\begin{defn} 
[Definition \ref{def:DefinitionUFH-LPQ},  cf. Naruse \cite{Naruse2017}]  
For a sequence of positive integers 
 $\lambda = (\lambda_{1}, \ldots, \lambda_{r})$ with $r  \leq n$,  
we define 
\begin{equation*}   
\begin{array}{lll} 
   HP^{\L}_{\lambda} (\x_{n}; t|\b)  
      :=     \displaystyle{\sum_{\overline{w}  \in S_{n}/(S_{1})^{r} \times S_{n-r}}}   
          w \cdot   \left [  
                                  [\x|\b]_{\L}^{\lambda}   \prod_{i=1}^{r}   
                                                     \prod_{j = i + 1}^{n}  
                                                      \dfrac{x_{i} +_{\L}  [t]  (\overline{x}_{j})} 
                                                                 {x_{i} +_{\L}  \overline{x}_{j} } 
                        \right ],   
\medskip    \\   
 HQ^{\L}_{\lambda} (\x_{n}; t|\b)  
      :=     \displaystyle{\sum_{\overline{w}  \in S_{n}/(S_{1})^{r} \times S_{n-r}}}   
          w \cdot   \left [  
                                  [[\x; t|\b]]_{\L}^{\lambda}   \prod_{i=1}^{r} 
                                                     \prod_{j = i + 1}^{n}  
                                                      \dfrac{x_{i} +_{\L}  [t] (\overline{x}_{j})} 
                                                                 {x_{i} +_{\L}  \overline{x}_{j} } 
                        \right ].    
\end{array} 
\end{equation*} 
\end{defn}  
\noindent
To the best of our knowledge,  
even  a factorial version of the ordinary Hall--Littlewood 
polynomials  has not appeared in the literature. 
Here we emphasize the importance of  these factorial  Hall--Littlewood  polynomials. 
 In fact, they  will be needed in describing the torus-equivariant cohomology 
of  {\it $p$-compact flag variety} corresponding to $G(e, 1, n)$
(cf. recent work of Ortiz \cite{Ortiz2015}).  
In this context,  the ``deformation parameters''  $\bm{b}$ are interpreted as the 
torus-equivariant parameters.
We will discuss this new aspect of the Hall--Littlewood functions 
in more detail  
in our forthcoming paper \cite{Nakagawa-Naruse2022}.

Then,   we show that our factorial Hall--Littlewood $P$- and $Q$-functions 
have the so-called {\it vanishing property} 
 (see Propositions \ref{prop:VanishingProperty(Cohomology)},  \ref{prop:VanishingProperty(Cobordism)}).  
We emphasize that this vanishing property will be useful in describing the so-called 
{\it GKM description} of the torus-equivariant cohomology ring of  the $p$-compact 
flag variety corresponding to $G(e, 1, n)$ (\cite{Nakagawa-Naruse2022}). 
By the vanishing property,  we can derive a Pieri-type formula for factorial 
Hall--Littlewood $P$-polynomials (see Proposition \ref{prop:Pieri-typeFormulaFH-LP}). 
Moreover, a simple recursive argument based on the associativity of factorial Hall--Littlewood 
$P$-polynomials, we can derive  a certain generalization of the hook formula 
(see Proposition \ref{prop:OurHookFormula}).  
We then   give a characterization of  them  in terms of Gysin maps from full flag bundles 
    in the complex cobordism theory (Proposition \ref{prop:CharacterizationUFH-LPQ}). 
 Using this characterization, we  can derive  {\it generating functions} for 
the universal factorial Hall--Littlewood $P$- and $Q$-functions. 
The idea of getting   our  result is to apply the Gysin formula 
for a projective bundle repeatedly to the full flag bundle 
since a full flag bundle is constructed as a sequence of projective bundles. 
However, the existence of the  deformation parameter
$\b = (b_{1}, b_{2},\ldots)$ prevent us from a direct application of the Gysin formula.  
To circumvent this difficulty, we developed  a specific modification in each step
(for details, see \S \ref{sec:GFUFH-LPQ}).  
Then, carrying out an argument carefully, we succeeded in getting  
the required result.  
For a sequence of positive integers  $\lambda = (\lambda_{1}, \ldots, \lambda_{r})$ with 
$r \leq n$, we set  
\begin{equation*} 
\begin{array}{llll} 
  &  \widetilde{\mathcal{HP}}^{\L, (n)}_{i, \lambda_{i}}(u_{1}, u_{2},  \ldots, u_{i}|\bm{b})
   :=   \dfrac{u_{i}} {u_{i} +_{\L} [t] (\ou_{i})} 
                  \cdot  \dfrac{1}{\mathscr{P}^{\L} (u_{i})}     \medskip \\
 &   \hspace{3.3cm}   \times \left (  
             \displaystyle{\prod_{j=1}^{n}}   \dfrac{u_{i} +_{\L}  [t] (\ox_{j})} 
                                                                        {u_{i} +_{\L}  \ox_{j}} 
                   \prod_{j=1}^{i-1}   \dfrac{u_{i} +_{\L}  \ou_{j}}  {u_{i} +_{\L}  [t]  (\ou_{j})}
                      \prod_{j=1}^{\lambda_{i}} \dfrac{u_{i} +_{\L}  b_{j}} {u_{i}}   
                               -                   
                                         t^{n  - i + 1}  
                                        \prod_{j=1}^{\lambda_{i}} \dfrac{b_{j}} {u_{i}} 
                                              \right ),      \medskip  \\ 
  &     \widetilde{\mathcal{HP}}^{\L, (n)}_{\lambda}  (\bm{u}_{r}|\bm{b}) 
       =    \widetilde{\mathcal{HP}}^{\L, (n)}_{\lambda} (u_{1},  u_{2}, \ldots, u_{r}|\bm{b}) 
      :=   \displaystyle{\prod_{i=1}^{r}}    
                \widetilde{\mathcal{HP}}^{\L, (n)}_{i, \lambda_{i}}(u_{1}, u_{2}, \ldots, u_{i}|\bm{b}).  \medskip 
\end{array} 
\end{equation*} 
Then,  our main result in this paper is given  as follows: 
\begin{theorem} [Theorem \ref{thm:GFUFH-LP}]    
The universal factorial Hall--Littlewood $P$-function 
$HP^{\L}_{\lambda}(\x_{n}; t|\b)$ is  the 
coefficient of $\bm{u}^{-\lambda} =  u_{1}^{-\lambda_{1}} u_{2}^{-\lambda_{2}} \cdots u_{r}^{-\lambda_{r}}$ 
in $\widetilde{\mathcal{HP}}^{\L, (n)}_{\lambda}(u_{1},  u_{2}, \ldots, u_{r}|\bm{b})$.  Thus 
\begin{equation*} 
         HP^{\L}_{\lambda}(\x_{n}; t|\b) 
    =  [\bm{u}^{-\lambda}]  \left ( 
                                                \widetilde{\mathcal{HP}}^{\L, (n)}_{\lambda}   (\bm{u}_{r}|\bm{b})  
                                  \right ). 
\end{equation*}   
\end{theorem}  
\noindent
Using similar, but simpler technique, we can also obtain a generating function 
for $HQ^{\L}_{\lambda}(\bm{x}_{n}; t|\bm{b})$ (see Theorem \ref{thm:GFUFH-LQ}).   
Here we stress the usefulness of a technique of generating functions. 
For instance, it is easy to derive 
Pfaffian formulas for factorial $K$-theoretic $Q$-polynomials in a simple and uniform manner  
(see Theorem \ref{thm:PfaffianFormulaGQ}).  
Moreover,  a certain {\it cancellation property} (cf.  Pragacz \cite[\S 2]{Pragacz1991}) 
of the factorial Hall--Littlewood 
$P$- and $Q$-polynomials can be verified immediately  (see Proposition \ref{prop:e-CancellationProperty}).     
For further applications of generating functions such as the so-called 
Pieri rule for $K$-theoretic $P$- and $Q$-polynomials, 
see also Naruse \cite{Naruse2017}.

\subsection{Organization of the paper}
The paper is organized as follows:  In Section \ref{sec:NotationConventions},  
we prepare notation and conventions 
concerning  the universal formal group law,  
a Gysin formula for a projective bundle,  which will be used throughout the paper.  
In Section \ref{sec:UFH-LPQ},  the universal factorial 
Hall--Littlewood $P$- and $Q$-functions are  introduced, and 
a characterization of them by means of  a Gysin map is given.  
The vanishing property of these functions are also discussed. 
By the vanishing property, a Pieri-type formula and a generalization of the hook formula 
are derived. 
 Using Gysin formulas for flag bundles  
and  characterizations of  
the  Hall--Littlewood functions by means of 
Gysin maps,  in   Section \ref{sec:GFUFH-LPQ},    
we  obtain generating functions for   these 
universal factorial Hall--Littlewood  functions.  
In Section \ref{sec:ApplicationGF},  
using our generating functions, we  shall show that the factorial 
Hall--Littlewood $P$- and $Q$-polynomials satisfy certain 
cancellation property. 
Pfaffian formulas for  factorial  
 $K$-theoretic $Q$-polynomials    can be obtained  as a by-product.    
In Appendix (Section \ref{sec:Appendix}),  we deal with   the topic 
closely related to the current work, namely, 
generating functions for  the {\it dual Grothendieck polynomials}
and the {\it dual $K$-theoretic Schur  $Q$-polynomials}.

\vspace{0.3cm}  

\textbf{Acknowledgments.} \quad    
The  first author is partially supported by the Grant-in-Aid for Scientific  Research 
          (C)  18K03303,   Japan Society for  the Promotion of Science. 
   The second author is partially supported by the Grant-in-Aid for Scientific  Research  (B) 16H03921, 
          Japan Society for  the Promotion of Science.   
Finally, the authors dedicate this work to the memory of  Professor Piotr Pragacz
who sadly passed away on October 25th, 2022.  Our work is greatly  inspired by his 
foundational works \cite{Pragacz1991}, \cite{Pragacz2015}.

\section{Notation, conventions, and  preliminary results
}  \label{sec:NotationConventions} 
For notation and conventions, we shall follow those used in our previous papers 
\cite{Nakagawa-Naruse2018},  \cite{Nakagawa-Naruse2019(arXiv)}.  
However,  to make the exposition self-contained  as much as possible, 
 we collect some of them  frequently  used in this paper.  
\subsection{Lazard ring $\L$ and the universal formal group law $F_{\L}$}    \label{subsec:UFGL}   
Let 
\begin{equation*} 
F_{\L} (u,v) =   u + v + \sum_{i,j \geq 1} a^{\L}_{i,j} u^{i}  v^{j}   \in    \L[[u,v]] 
\end{equation*} 
be the {\it universal formal group law}, where  $\L$ is   the {\it Lazard ring}.  
Namely, $F_{\L} (u, v)$  
is a formal power series in two indeterminates $u$,  $v$ with coefficients $a^{\L}_{i,j} \in \L$ 
which satisfies the axioms of the  formal group law.   
For the universal formal group law, we shall use the following notation: 
\begin{equation*} 
\begin{array}{llll} 
  &  u  +_{\L}    v =  F_{\L}(u,  v)   \quad  & \text{(formal sum)}, \medskip \\ 
  &  \overline{u} =   [-1]_{\L} (u)  = \chi_{_{\L}}(u)  & \text{(formal  inverse of} \;   u),  \medskip  \\
  &   u -_{\L} v =  u +_{\L}  [-1]_{\L}(v) = u +_{\L}  \overline{v}   &  \text{(formal subtraction)}. 
\medskip  
\end{array}
\end{equation*}
Furthermore,     we define $[0]_{\L}(u) := 0$, and inductively,  
   $[n]_{\L}(u)  :=  [n-1]_{\L}(u) +_{\L} u$
for a positive integer  $n \geq 1$.  We also define   
 $[-n]_{\L}(u) := [n]_{\L}([-1]_{\L}(u))$ for $n \geq 1$.     
We call $[n]_{\L}(u)$ the {\it $n$-series} in the sequel.    
Denote by  $\ell_{\L}  (u)   \in \L \otimes \Q [[u]]$ the  {\it logarithm} 
of  $F_{\L}$, i.e.,  a unique formal power series with leading term $u$ 
such that   
\begin{equation*} 
\ell_{\L}(u +_{\L} v)  =   \ell_{\L} (u) + \ell_{\L} (v).   
\end{equation*} 
Using the logarithm $\ell_{\L}(u)$, one can rewrite the $n$-series $[n]_{\L}(u)$ for a non-negative integer $n$   as  
$\ell_{\L}^{-1}  (n \cdot \ell_{\L} (u))$, 
where $\ell_{\L}^{-1}(u)$ is the formal power series  inverse to $\ell_{\L}(u)$.   
This formula allows us to define 
\begin{equation*} 
   [t]_{\L}(x)  =   [t](x)  :=  \ell_{\L}^{-1} (t  \cdot \ell_{\L} (x))
\end{equation*} 
 for an  indeterminate $t$.  This is a natural extension of  $t \cdot x$ as well as 
the $n$-series  $[n]_{\L} (x)$.

Next we shall introduce  various generalizations of  the ordinary power 
of variables.    
 Let $\bm{x} = (x_{1}, x_{2}, \ldots)$ be a countably infinite  sequence of independent 
variables.    We also introduce  another set of  independent variables $\bm{b} = (b_{1}, b_{2}, \ldots)$. 
 Then, for a positive  integer $k\geq 1$, we define a generalization of the 
ordinary $k$-th power $x^{k}$  of one   variable $x$ 
 by 
\begin{equation*} 
   [x| \b]_{\L}  ^{k}  := \displaystyle{\prod_{j=1}^{k}}  (x  +_{\L}   b_{j})  
               = (x +_{\L}   b_{1})(x +_{\L}   b_{2}) \cdots (x +_{\L}   b_{k}). 
\end{equation*}   
We set $[x |\b]^{0}_{\L}   := 1$.  
 For a sequence of positive integers
 $\lambda = (\lambda_{1}, \ldots, \lambda_{r})$,
we set 
\begin{equation*} 
    [\x| \b ]_{\L}^{\lambda}  
:=  \displaystyle{\prod_{i = 1}^{r} }  [x_i|\b]_{\L}^{\lambda_{i}} 
 = \prod_{i=1}^{r}  \prod_{j=1}^{\lambda_{i}} (x_{i} +_{\L} b_{j}).  
\end{equation*} 
Similarly, we define  
\begin{equation*} 
   [[x| \b  ]]_{\L}^{k}  :=(x +_{\L}  x)[x| \b ]_{\L}^{k-1}  
   = (x +_{\L}   x) (x  +_{\L}   b_{1})(x +_{\L}   b_{2}) \cdots (x +_{\L}   b_{k-1}).    
\end{equation*}   
For a sequence of positive integers   $\lambda = (\lambda_{1},  \ldots, \lambda_{r})$,   
we set  
\begin{equation*}
[[\x| \b ]]_{\L}^{\lambda}  
    :=  \displaystyle{\prod_{i = 1}^{r} } [[x_{i}|\b]]_{\L}^{\lambda_{i}} 
    =  \prod_{i=1}^{r}  (x_{i} +_{\L} x_{i}) [x_{i}|\b]_{\L}^{\lambda_{i} -1}.
\end{equation*}  
Moreover, for  indeterminates $x$ and $t$,  we define   
\begin{equation*} 
    [[x; t|\b]]_{\L}^{k}  :=  (x +_{\L}  [t](\overline{x}))  [x|\b]_{\L}^{k-1}
\end{equation*} 
for a positive integer $k \geq 1$.  
 For a sequence of positive integers   $\lambda = (\lambda_{1}, \ldots, \lambda_{r})$, 
we define  
\begin{equation*} 
    [[\x ; t|\b]]_{\L}^{\lambda}   
    :=  \prod_{i=1}^{r}  [[x_{i}; t|\b]]_{\L}^{\lambda_{i}} 
   = \prod_{i=1}^{r}  (x_{i} +_{\L}  [t](\overline{x}_{i}) [x_{i}|\b]_{\L}^{\lambda_{i} -1}.   
\end{equation*}

\subsection{
        Gysin  formula for  a projective bundle  in complex cobordism
}   \label{subsec:GysinFormulaProjectiveBundle(ComplexCobordism)}  
Recall from Quillen \cite[Theorem 1]{Quillen1969}  the Gysin formula  for a 
projective bundle in complex cobordism.  We shall state his result in a manner 
suitable for our purpose (for more details, see 
Nakagawa--Naruse \cite[\S 3.1]{Nakagawa-Naruse2019(arXiv)}):  
Let $E \longrightarrow X$ be a complex vector bundle of rank $n$.  
For any integer $m \in \Z$, denote by $\mathscr{S}^{\L}_{m}(E) = \mathscr{S}^{MU}_{m} (E)$
the  Segre class of $E$ in complex cobordism,    and 
\begin{equation*} 
    \mathscr{S}^{\L} (E; u)  :=  \sum_{m \in \Z} \mathscr{S}^{\L}_{m}(E)  u^{m} 
\end{equation*} 
its Segre series.  The explicit expression of  $\mathscr{S}^{\L}(E; u)$ is 
given by 
\begin{equation}   \label{eqn:SegreSeries(ComplexCobordism)}  
     \mathscr{S}^{\L}(E; u)  
        = \left.    \dfrac{1}{\mathscr{P}^{\L} (z)}   
                             \prod_{j=1}^{n}  \dfrac{z}{ z +_{\L}  \overline{x}_{j}}   
                                       \right |_{z = u^{-1}} 
                 =   \left.    \dfrac{1}{\mathscr{P}^{\L} (z)}   
                                   \dfrac{z^{n}}{\prod_{j=1}^{n} (z +_{\L}  \overline{x}_{j}) }   
                       \right |_{z = u^{-1}},    
\end{equation}
where  $\mathscr{P}^{\L}(z) :=  1 + \sum_{i=1}^{\infty} a^{\L}_{i, 1} z^{i}$, 
and $x_{1}, \ldots, x_{n}$ are the Chern roots of $E$ in complex cobordism.

Now consider the Grassmann bundle  $\pi^{1}:  G^{1}(E) \longrightarrow X$ of  {\it hyperplanes} in $E$.
Denote by  $Q^{1}$ the tautological {\it quotient}  bundle on $G^{1}(E)$.  
Put $x_{1} := c_{1}^{MU}(Q^{1})  \in MU^{2}(G^{1}(E))$.  
For a monomial $m$ of a formal Laurent series $F$, 
  we denote by  $[m](F)$ the coefficient of $m$ in $F$. 
Note that the Grassmann bundle $G^{1}(E)$ of hyperplanes in $E$ is canonically 
isomorphic to the projective bundle $P(E^{\vee}) = G_{1}(E^{\vee})$ of lines in 
the dual bundle $E^{\vee}$. Then, by dualizing the formula \cite[(3.4)]{Nakagawa-Naruse2019(arXiv)}, 
we have the following form of Quillen's Gysin formula:  
\begin{prop}     \label{prop:FundamentalFormula(TypeA)(ComplexCobordism)} 
For a polynomial $f(u) \in MU^{*}(X)[u]$, the Gysin map 
 $\pi^{1}_{*}:  MU^{*}(G^{1}(E))    \\ \longrightarrow MU^{*}(X)$
is described by the following formula$:$  
\begin{equation}    \label{eqn:FundamentalFormula(TypeA)(ComplexCobordism)}
    \pi^{1}_{*} (f(x_{1})) 
                               =   [u^{n-1}]  (f(u)  \cdot  \mathscr{S}^{\L} (E;  1/u)).    
\end{equation}  
\end{prop}   
\noindent
This is the fundamental formula for establishing more general Gysin 
formulas for  general  flag bundles.

Here we shall fix some notation concerning flag bundles:   
Let $E \longrightarrow X$ be a complex vector bundle of rank $n$. 
For  a positive integer $r = 1, 2, \ldots, n$, denote by 
$\pi^{r, r-1, \ldots, 1}:  \F \ell^{r, r-1, \ldots, 1}(E)  
=  \F \ell_{n-r, n-r + 1, \ldots, n-1} (E)\longrightarrow X$ 
be the associated flag bundle. 
Thus a point in $\F \ell^{r, r-1, \ldots, 1}(E)$ is written as a pair $(x, (W_{\bullet})_{x})$, 
where  $(W_{\bullet})_{x}$ is a flag, i.e., nested subspaces 
of the form $(W_{1})_{x}    \subset  (W_{2})_{x}    \subset \cdots \subset  (W_{r})_{x}$, 
  $\mathrm{codim} \, (W_{i})_{x}  = r + 1 -i$, 
in the fiber $E_{x}$ of $E$ over  each  point $x \in X$.    
We shall call the flag 
bundle of the form $\pi^{r, r-1, \ldots, 1}: \F \ell^{r, r-1, \ldots, 1}(E) \longrightarrow X$ the  {\it full} flag bundle.
  When $r = n$, we call  
$\pi^{n, n-1, \ldots, 1}: \F \ell^{n, n-1, \ldots, 1} (E) \longrightarrow X$ the {\it complete}  flag bundle,
and just write $\pi:  \F \ell  (E) \longrightarrow X$.  
On $\F \ell  (E)$, there is the universal flag of subbundles 
\begin{equation*} 
    0 = U_{0}  \subset U_{1} \subset  \cdots  \subset U_{i} \subset \cdots 
 \subset U_{n-1}   \subset  U_{n}  = \pi^{*}(E), 
\end{equation*} 
where $\rank U_{i} = i \; (i = 0, 1, \ldots, n)$.  
and we put 
\begin{equation}   \label{eqn:ChernRootsE(ComplexCobordism)}
    x_{i}  :=   c^{MU}_{1}(U_{n + 1 -i}/U_{n-i}) \in MU^{2}(\F \ell  (E)) \quad (i = 1, 2, \ldots, n), 
\end{equation} 
which  are the $MU^{*}$-theory Chern roots of $E$.   
It is well-known 
that the full flag bundle $\F \ell^{r, r-1, \ldots, 1}(E)$ is constructed as a sequence of 
of  Grassmann bundles of  codimension one hyperplanes$:$ 
\begin{equation}   \label{eqn:ConstructionFlagBundle}  
 \pi^{r, r-1, \ldots, 1}:    \F \ell^{r, r-1, \ldots, 2, 1} (E)  = G^{1}(U_{n - r + 1})  
   \overset{\pi^{r}}{\longrightarrow}   \cdots  \longrightarrow G^{1}(U_{n-1})  
   \overset{\pi^{2}}{\longrightarrow}  G^{1}(E)    
  \overset{\pi^{1}}{\longrightarrow}  X. 
\end{equation}

\section{Universal factorial Hall--Littlewood $P$- and $Q$-functions}   \label{sec:UFH-LPQ}  
In this section, we shall introduce our main object to study, 
the {\it universal factorial Hall--Littlewood $P$- and 
$Q$-functions},  which are universal  as well as factorial analogues of the ordinary Hall--Littlewood polynomials.

\subsection{Universal factorial Hall--Littlewood $P$- and $Q$-functions}   \label{subsec:UFH-LPQ}  
\subsubsection{Definition of the universal factorial Hall--Littlewood $P$- and $Q$-functions} 

We shall use  the notation  introduced in \S 
\ref{subsec:UFGL}.  
We provide the variables $\bm{x} = (x_{1}, x_{2}, \ldots)$ and $\bm{b} = (b_{1}, b_{2}, \ldots) $
with $\deg \, (x_{i}) = \deg \, (b_{i}) = 1$ for $i = 1, 2, \ldots$.  
Then we make the following definition:  
\begin{defn} [Universal factorial  Hall--Littlewood $P$- and $Q$-functions] 
   \label{def:DefinitionUFH-LPQ}   
For a sequence of positive integers 
 $\lambda = (\lambda_{1}, \ldots, \lambda_{r})$ with $r \leq n$,   
we define 
\begin{equation*}   \label{eqn:DefinitionUFH-LPQ}   
\begin{array}{lll} 
   HP^{\L}_{\lambda} (\x_{n}; t|\b)  
      :=     \displaystyle{\sum_{\overline{w}  \in S_{n}/(S_{1})^{r} \times S_{n-r}}}   
          w \cdot   \left [  
                                  [\x|\b]_{\L}^{\lambda}   \prod_{i=1}^{r}   
                                                     \prod_{j = i + 1}^{n}  
                                                      \dfrac{x_{i} +_{\L}  [t]  (\overline{x}_{j})} 
                                                                 {x_{i} +_{\L}  \overline{x}_{j} } 
                        \right ],   
\medskip    \\   
 HQ^{\L}_{\lambda} (\x_{n}; t|\b)  
      :=     \displaystyle{\sum_{\overline{w}  \in S_{n}/(S_{1})^{r} \times S_{n-r}}}   
          w \cdot   \left [  
                                  [[\x; t|\b]]_{\L}^{\lambda}   \prod_{i=1}^{r} 
                                                     \prod_{j = i + 1}^{n}  
                                                      \dfrac{x_{i} +_{\L}  [t] (\overline{x}_{j})} 
                                                                 {x_{i} +_{\L}  \overline{x}_{j} } 
                        \right ].    
\end{array} 
\end{equation*} 
We also define
\begin{equation*} 
    HP^{\L}_{\lambda}(\x_{n}; t) := HP^{\L}_{\lambda}(\x_{n}; t|\bm{0}) \quad 
   \text{and}  \quad 
    HQ^{\L}_{\lambda}(\x_{n}; t) := HQ^{\L}_{\lambda}(\x_{n}; t|\bm{0}). 
\end{equation*} 
In what follows,  
$HP^{\L}_{\lambda}(\x_{n}; t)$ and $HQ^{\L}_{\lambda}(\x_{n}; t)$ will be called 
the {\it universal Hall--Littlewood $P$- and $Q$-functions} respectively. 
\end{defn}  
\noindent
It follows immediately  from   Definition \ref{def:DefinitionUFH-LPQ}  that when $t = -1$, then
$HP^{\L}_{\lambda}(\x_{n}; -1|\b)$ (resp. $HQ^{\L}_{\lambda}(\x_{n}; -1|\b)$)
coincides with the universal factorial Schur $P$-function
$P^{\L}_{\lambda}(\x_{n}|\b)$ (resp.  $Q$-function $Q^{\L}_{\lambda}(\x_{n}|\b)$), 
for a strict partition $\lambda$,  which have been introduced 
in our previous paper \cite[Definition 4.1]{Nakagawa-Naruse2016}.    
In contrast to this,  
 when $t = 0$, both  $HP^{\L}_{\lambda}(\bm{x}_{n}; 0|\bm{b})$ and
 $HQ^{\L}_{\lambda} (\bm{x}_{n}; 0 |\bm{b})$ are different from 
the universal factorial Schur functions 
$s^{\L}_{\lambda}(\bm{x}_{n}|\bm{b})$ (\cite[Definition 4.10]{Nakagawa-Naruse2016}), 
 $\mathbb{S}^{\L}_{\lambda}(\bm{x}_{n}|\bm{b})$ (\cite[Definition 5.1]{Nakagawa-Naruse2018}).

\subsubsection{Factorial Hall--Littlewood $P$- and $Q$-polynomials} 
 The specialization from $F_{\L}(u, v) = u +_{\L} v$ to $F_{a}(u, v) = u + v$ is
of particular importance.  Under this specialization, the generalized powers $[x|\b]_{\L}^{k}$, 
$[[x; t|\b]]_{\L}^{k}$ reduce to $[x|\b]^{k}
 = \prod_{j=1}^{k} (x + b_{j})$, 
$[[x; t|\b]]^{k} = (x - tx) [x|\b]^{k-1}$
respectively, and we obtain 
{\it new}  symmetric polynomials
denoted by $HP_{\lambda}(\x_{n}; t |\b)$ and $HQ_{\lambda}(\x_{n}; t|\b)$ respectively.   
More explicitly, these  are defined as follows:  
\begin{defn} [Factorial Hall--Littlewood $P$- and $Q$-polynomials]  \label{defn:DefinitionFH-LPQ}  
For a sequence of positive integers $\lambda = (\lambda_{1}, \ldots, \lambda_{r})$ 
with $r \leq n$, 
we define  
\begin{equation*}   \label{eqn:DefinitionFH-LPQ}   
\begin{array}{lll} 
   HP_{\lambda} (\x_{n}; t|\b)  
      & := 
       \displaystyle{\sum_{\overline{w}  \in S_{n}/(S_{1})^{r} \times S_{n-r}}}   
          w \cdot   \left [  
                                  [\x|\b]^{\lambda} 
                                         \prod_{i=1}^{r}   
                                                     \prod_{j = i + 1}^{n}  
                                                      \dfrac{x_{i}  -tx_{j}} 
                                                                 {x_{i} - x_{j} } 
                        \right ],   
\medskip    \\   
      & =     \displaystyle{\sum_{\overline{w}  \in S_{n}/(S_{1})^{r} \times S_{n-r}}}   
          w \cdot   \left [  
                                  \prod_{i=1}^{r} \prod_{j=1}^{\lambda_{i}} (x_{i} + b_{j}) 
                                            \times   \prod_{i=1}^{r}   
                                                     \prod_{j = i + 1}^{n}  
                                                      \dfrac{x_{i}  -tx_{j}} 
                                                                 {x_{i} - x_{j} } 
                        \right ],   
\medskip    \\   
 HQ_{\lambda} (\x_{n}; t|\b)  
     & :=  \displaystyle{\sum_{\overline{w}  \in S_{n}/(S_{1})^{r} \times S_{n-r}}}   
          w \cdot   \left [  
                                  [[\bm{x}; t|\b]]^{\lambda} 
                                         \prod_{i=1}^{r} 
                                                     \prod_{j = i + 1}^{n}  
                                                      \dfrac{x_{i} -t x_{j}} 
                                                                 {x_{i}  - x_{j} } 
                        \right ]  \medskip \\
     &=   (1-t)^{r}   \displaystyle{\sum_{\overline{w}  \in S_{n}/(S_{1})^{r} \times S_{n-r}}}   
          w \cdot   \left [  
                                  \prod_{i=1}^{r} \prod_{j=1}^{\lambda_{i}-1} x_{i}(x_{i} + b_{j}) 
                                             \times  \prod_{i=1}^{r} 
                                                     \prod_{j = i + 1}^{n}  
                                                      \dfrac{x_{i} -t x_{j}} 
                                                                 {x_{i}  - x_{j} } 
                        \right ].    
\end{array} 
\end{equation*} 
We also define 
\begin{equation*} 
    HP_{\lambda}(\x_{n}; t) := HP_{\lambda}(\x_{n}; t|\bm{0}) \quad 
   \text{and}  \quad 
    HQ_{\lambda}(\x_{n}; t) := HQ_{\lambda}(\x_{n}; t|\bm{0}),  
\end{equation*} 
and will be called the Hall--Littlewood $P$- and $Q$-polynomials respectively.
\end{defn}  
\noindent
Note that, by definition, we have $HQ_{\lambda}(\x_{n}; t|\b) 
= (1-t)^{\ell (\lambda)} HP_{\lambda}(\x_{n}; t|0, \bm{b})$.    
For a strict partition $\lambda$, 
if $t$ specializes to be $-1$, then $HP_{\lambda}(\bm{x}_{n}; -1 |\bm{b})$ 
and $HQ_{\lambda}(\bm{x}_{n}; -1 |\bm{b})
 = 2^{\ell (\lambda)}  HP_{\lambda}(\bm{x}_{n};  -1|0, \bm{b})$ coincide with 
the factorial Schur $P$- and $Q$-polynomials 
(by 
replacing $\bm{b}$  with  $-\bm{b} = (-b_{1}, -b_{2}, \dots)$) (for their definition, 
see Ikeda--Mihalcea--Naruse \cite[\S 4.2]{IMN2011}).  
However, for a partition $\lambda$, both $HP_{\lambda}(\bm{x}_{n}; 0|\bm{b})$ 
and $HQ_{\lambda}(\bm{x}_{n}; 0|\bm{b})$ do not coincide with 
the factorial Schur polynomial (for its definition, see Molev--Sagan \cite[\S 2, (3)]{Molev-Sagan1999}).

\begin{ex}     \label{ex:FH-LPP} 
Direct computation from Definition $\ref{defn:DefinitionFH-LPQ}$
gives  some examples$:$ 
\begin{equation*} 
  \begin{array}{lll} 
       HP_{(1)}(\bm{x}_{n}; t|\bm{b})  =  x_{1} + x_{2} + \cdots + x_{n}  +  \dfrac{1-t^{n}}{1-t} b_{1}, \medskip \\
       HP_{(1^2)}(\bm{x}_{n}; t|\bm{b}) 
   =  (1 + t) \left [ m_{(1^2)}(\bm{x}_{n})  +  \dfrac{1-t^{n-1}}{1-t} b_{1} m_{(1)}(\bm{x}_{n}) 
            +  \dfrac{(1-t^{n-1})(1-t^{n})}{(1-t)(1-t^2)}  b_{1}^{2}  \right ], \medskip \\
    HP_{(2)}(\bm{x}_{n}; t|\bm{b})  = (s_{(2)}(\bm{x}_{n}) - t s_{(1^2)} (\bm{x}_{n}) 
                                            + (b_{1} + b_{2})s_{(1)}(\bm{x}_{n}) + b_{1}b_{2} \dfrac{1-t^m}{1-t}.  \medskip 
\end{array}  
\end{equation*}
Here $m_{\lambda}(\bm{x}_{n})$ and $s_{\lambda}(\bm{x}_{n})$ are respectively 
 the monomial 
symmetric polynomials and Schur polynomials corresponding to 
$\lambda$.  
\end{ex}

If $\lambda$ is a partition of length $\ell (\lambda) = r \leq n$, 
i.e.,  $\lambda_{1} \geq \lambda_{2} \geq \cdots \geq \lambda_{r} > 0$, 
our factorial Hall--Littlewood $P$- and $Q$-polynomials are 
related to  Macdonald's  Hall--Littlewood $P$- and $Q$-polynomials 
in the following way: 
We rewrite $\lambda$   as
   $\lambda = (n_{1}^{p_{1}} \; n_{2}^{p_{2}} \; \cdots n_{d-1}^{p_{d-1}} \; n_{d}^{p_{d}})$, 
where $n_{1} > n_{2} > \cdots > n_{d-1} > n_{d} = 0$, each $p_{i} > 0$, $p_{d}= n-r$, and  
$\sum_{i=1}^{d} p_{i} = n$.  We put $\nu (k) := \sum_{i=1}^{k} p_{i}$ for $k = 1, \ldots, d$
and $\nu (0) := 0$.  Denote by $S_{p_{k}}$ the symmetric group on 
$m_{k}$ letters $\nu (k-1) + 1, \ldots , \nu (k)$ for $k = 1, \ldots, d$. 
Thus the stabilizer subgroup $S_{n}^{\lambda}$  of $\lambda$ under the action of 
$S_{n}$ on $\lambda$ is given by $S_{n}^{\lambda} =  \prod_{k=1}^{d}  S_{p_{k}}$.  
For an integer $k \geq 0$, let $v_{k}(t) :=  \prod_{i=1}^{k}  \frac{1-t^{i}}{1-t}$, 
and for the above partition $\lambda$,  
we define\footnote{
Do not confuse $v_{\lambda > 0}(t)$ with $v_{\lambda}(t) := \prod_{i \geq 0}v_{m_{i}}(t)$ 
in Macdonald \cite[Chapter III, \S 1]{Macdonald1995},  where $m_{i} = m_{i}(\lambda)$ 
means the {\it multiplicity} for each $i \geq 0$.   
}
\begin{equation*} 
   v_{\lambda > 0} (t) :=  \displaystyle{\prod_{k=1}^{d-1}
                                                  }  v_{p_{k}} (t). 
\end{equation*}  
Using the identity 
\begin{equation}  \label{eqn:IdentityMacdonald}   
         \sum_{w \in S_{n}}  w \cdot  \left [   \prod_{1 \leq i  < j \leq n} 
                 \dfrac{x_{i} - tx_{j}} {x_{i} - x_{j}}  
                                            \right ] 
      = v_{n}(t)
\end{equation} 
in \cite[Chapter III, (1.4)]{Macdonald1995}, 
one  can prove  the following fact  along the same line as  the case of 
the usual Hall--Littlewood 
polynomials (\cite[Chapter III, (1.5)]{Macdonald1995}):  
\begin{equation}   \label{eqn:DivisibleByv_lambda(t)} 
    HP_{\lambda}(\bm{x}_{n}; t|\b)
    =    v_{\lambda > 0} (t)    \times 
             \displaystyle{\sum_{\overline{w}  \in S_{n}/S_{n}^{\lambda}
                                      } 
                             }  
           w  \cdot \left  [ 
                                 [\bm{x}|\bm{b}]^{\lambda} 
                                   \cdot 
                                       \prod_{\substack{ 
                                                                 1 \leq i < j \leq n \\
                                                                    \lambda_{i} > \lambda_{j } 
                                                            } 
                                                } 
                                             \dfrac{x_{i} - tx_{j}} {x_{i} - x_{j}}  
                         \right ].  
\end{equation} 
Thus $HP_{\lambda}(\bm{x}_{n}; t |\bm{b})$ is  divisible by $v_{\lambda > 0} (t)$. 
Taking this fact into account, we define 
\begin{equation}   \label{eqn:DefinitionP_lambda(x_n;t|b)}  
  P_{\lambda}(\x_{n}; t | \b) :=  \dfrac{1}{v_{\lambda > 0} (t)}  HP_{\lambda}(\x_{n}; t |\b), 
\end{equation} 
or equivalently, 
\begin{equation}   \label{eqn:DefinitionP_lambda(x_n;t|b)2} 
      P_{\lambda}(\x_{n}; t | \b) :=  \displaystyle{\sum_{\overline{w}  \in S_{n}/S_{n}^{\lambda}
                                      } 
                             }  
           w  \cdot \left  [ 
                                 [\bm{x}|\bm{b}]^{\lambda} 
                                   \cdot 
                                       \prod_{\substack{ 
                                                                 1 \leq i < j \leq n \\
                                                                    \lambda_{i} > \lambda_{j } 
                                                            } 
                                                } 
                                             \dfrac{x_{i} - tx_{j}} {x_{i} - x_{j}}  
                         \right ].  
\end{equation} 
It is this polynomial that can be considered as  a factorial version of 
Macdonald's Hall--Littlewood $P$-polynomial $P_{\lambda}(\x_{n}; t)$.   
Putting $\bm{b} = \bm{0}$  in  (\ref{eqn:DefinitionP_lambda(x_n;t|b)}), we have 
$HP_{\lambda}(\x_{n} ; t)  =  v_{\lambda > 0} (t) P_{\lambda}(\x_{n}; t)$.  
In particular,  for  $\lambda$ strict,  $HP_{\lambda}(\x_{n};t)$ coincides with 
$P_{\lambda}(\x_{n};t)$.  
On the other hand, 
by the argument in Macdonald' book \cite[pp.210--211]{Macdonald1995}, 
we see that $HQ_{\lambda}(\x_{n}; t) $  equals to the ordinary Hall--Littlewood $Q$-polynomial
$Q_{\lambda}(\x_{n}; t)$.  
\begin{rem} 
\begin{enumerate} 
\item The universal analogue of the left hand side of $(\ref{eqn:IdentityMacdonald})$, 
namely, 
\begin{equation*} 
    \sum_{w \in S_{n}}  w \cdot \left [ 
                                                       \prod_{1 \leq i < j \leq n}   
                                                        \dfrac{x_{i} +_{\L} [t](\overline{x}_{j})} 
                                                                  {x_{i} +_{\L} \overline{x}_{j}} 
                                               \right ] 
\end{equation*} 
is no longer a polynomial in $t$ alone $($it contains the  variables $x_{1}, \ldots, x_{n}$$)$. 
 Therefore an analogous formula of 
 $(\ref{eqn:DivisibleByv_lambda(t)})$ does not hold in this case.

\item For a general sequence of positive integers $\lambda$, 
$HP_{\lambda}(\x_{n}; t|\b)$ may not be divisible by $v_{\lambda > 0} (t)$.  
\end{enumerate} 
\end{rem} 

\subsection{Characterization of the universal factorial  Hall--Littlewood $P$- and $Q$-functions} 
Geometrically, the universal factorial Hall--Littlewood  $P$- and $Q$-functions  are  characterized 
by means of  the Gysin map for certain flag bundles 
(We learned this idea from the work \cite{Pragacz2015} by Pragacz).   
Let $E \longrightarrow X$ be a complex vector bundle of rank $n$, and 
$x_{1}, \ldots, x_{n}$  are the $MU^{*}$-theory Chern roots of $E$ 
as in (\ref{eqn:ChernRootsE(ComplexCobordism)}).   
Consider the associated  flag bundle $\pi^{r, r-1, \ldots, 1}:  \F \ell^{r, r-1, \ldots, 1}(E) 
\longrightarrow X$.   Then,   it follows immediately from the above Definition 
\ref{def:DefinitionUFH-LPQ} and a description of the Gysin homomorphism 
$(\pi^{r, \ldots, 1})_{*}$  as a certain symmetrizing operator
 (Brion \cite[Proposition 1.1]{Brion1996}, 
Nakagawa--Naruse  \cite[Theorem 4.10]{Nakagawa-Naruse2018})
that  the following  formula holds:
\begin{prop} [Characterization of the universal factorial Hall--Littlewood $P$- and $Q$-functions] 
\label{prop:CharacterizationUFH-LPQ}   
\begin{align}  
    (\pi^{r, \ldots, 1})_{*} 
 \left  ( 
                [\x|\b]_{\L}^{\lambda}   \displaystyle{\prod_{i=1}^{r}}  
                  \prod_{j=i+1}^{n}  (x_{i} +_{\L}  [t](\overline{x}_{j}))   
\right  )  &=   HP^{\L}_{\lambda}(\x_{n} ; t|\b),    
 \label{eqn:CharacterizationUFH-LP}   \\
   (\pi^{r, \ldots, 1})_{*}  \left ( 
                                    [[\x ;t|\b]]_{\L}^{\lambda}   \displaystyle{\prod_{i=1}^{r}} 
               \prod_{j=i+1}^{n}  (x_{i} +_{\L}  [t](\overline{x}_{j} )) 
                      \right  )  &=  HQ^{\L}_{\lambda}(\x_{n} ; t|\b).  
  \label{eqn:CharacterizationUFH-LQ}     
\end{align}  
Here $\b = (b_{1}, b_{2}, \ldots)$ is a sequence of elements in $MU^{*}(X)$.   
\end{prop} 
\noindent
 This characterization  seems merely a  paraphrase of Definition \ref{def:DefinitionUFH-LPQ} 
at first sight.  However, this geometric interpretation 
 will be crucial in our current work.       
In fact,  as shown in the subsequent section, \S \ref{sec:GFUFH-LPQ}, 
a careful application of the fundamental Gysin formula
 (\ref{eqn:FundamentalFormula(TypeA)(ComplexCobordism)}) to the left hand 
side of (\ref{eqn:CharacterizationUFH-LP}), (\ref{eqn:CharacterizationUFH-LQ}) 
enables us to obtain the generating functions for the universal factorial 
Hall--Littlewood $P$- and  $Q$-functions.

\begin{rem} 
As a special case of the above result,  the factorial Hall--Littlewood $P$-polynomial 
$HP_{\lambda}(\x_{n}; t|\b)$ is characterized by the cohomology Gysin map, 
i.e., we have 
\begin{equation*} 
        (\pi^{r, \ldots, 1})_{*} 
 \left  ( 
                [\x|\b]^{\lambda}   \displaystyle{\prod_{i=1}^{r}}  
                  \prod_{j=i+1}^{n}  (x_{i}  - tx_{j})   
\right  )  =   HP_{\lambda}(\x_{n} ; t|\b). 
\end{equation*}   
A  factorial version of Macdonald's Hall--Littlewood $P$-polynomial 
$P_{\lambda}(\x_{n}; t|\b)$ can be also characterized by the Gysin map$:$  
Consider the partial flag bundle  $\pi^{\lambda}:  
\F \ell^{\lambda} (E) := \F  \ell^{\nu (d-1), \nu (d-2), \ldots, \nu (1)}(E)
\longrightarrow X$.     Here we write $\lambda 
= (n_{1}^{p_{1}}  \cdots n_{d}^{p_{d}})$  and $\nu (k) = \sum_{i=1}^{k} p_{i}$ 
 as in \S $\ref{subsec:UFH-LPQ}$.  
Then the following formula holds$:$ 
\begin{equation*} 
     (\pi^{\lambda})_{*} 
    \left  ( 
                [\x|\b]^{\lambda}   \displaystyle{\prod_{i=1}^{r}}  
                  \prod_{j=i+1}^{n}  (x_{i}  - tx_{j})   
\right  )  =   P_{\lambda}(\x_{n} ; t|\b). 
\end{equation*} 
\end{rem} 

\subsection{Vanishing properties of  factorial Hall--Littlewood $P$- and $Q$-polynomials
}
It is known that the factorial Schur  $S$-, $P$-, and $Q$-polynomials
have the remarkable property  
called   {\it vanishing property} (see Molev--Sagan \cite[Theorem 2.1]{Molev-Sagan1999}, 
Ivanov \cite[Theorem 5.3]{Ivanov2004}).  
In  this subsection, we shall show that 
our factorial Hall--Littlewood $P$- and $Q$-polynomials
have this property.    
Let $\b =(b_{1}, b_{2},\ldots)$ be a sequence of indeterminates, and
$t$ be an indeterminate.
For a partition $\mu =  (\mu_{1}, \mu_{2}, \ldots)$,  
let  $m_{i}  =m_{i}  (\mu)$ be the multiplicity of $i$ ($1 \leq i  \leq \mu_{1}$), i.e.,  
the number of components in $\mu$ whose size is equal to $i$.  
We  define
\begin{equation*} 
   -\b_\mu(t):=(-\b_{\mu_{1}}^{m_{\mu_{1}}}(t),   \ldots,   -\b_{2}^{m_{2}}(t),   -\b_{1}^{m_{1}}(t)), 
\end{equation*} 
where $-\b_{i}^{k}(t)  :=  (-b_{i},  -t b_{i}, \ldots,  -t^{k-1} b_{i})$ 
(we set $-\b_{i}^0(t)=(\; )$, the empty sequence).  
Let us consider to substitute the variables 
$\x_{n} = (x_{1}, \ldots, x_{n})$ with the sequence 
$-\b_{\mu}(t)$ for a partition $\mu$ of length $\ell (\mu) \leq n$. 
We sometimes write $\x_{n} \rightarrow -\b_{\mu}(t)$, or more specifically, 
say, $x_{1}  \rightarrow  -b_{\mu_{1}}$ when we make such   substitution.  
After the substitution $\x_{n} \rightarrow -\b_{\mu}(t)$ was made,  
denote by 
$\mathrm{ev}_{\mu} (x_{i})$ ($i=1,\ldots, n$)  the $i$-th entry of $-\b_{\mu} (t)$.  
Therefore  we have 
\begin{equation*} 
     (\mathrm{ev}_{\mu}(x_1),  \ldots,  \mathrm{ev}_{\mu}(x_n))  =  -\b_\mu(t). 
\end{equation*} 
We also use the notation $\mathrm{ev}_{\mu} ( f(x_{1},  \ldots, x_{n}))  = 
f(\mathrm{ev}_{\mu}(x_{1}), \ldots,  \mathrm{ev}_{\mu} (x_{n}))$  
in the following. 
For example,
if $\mu=(5,5,5,4,1,1)$,  then 
$m_1(\mu)=2$, $m_2(\mu)=0$, $m_3(\mu)=0$, $m_4(\mu)=1$, $m_5(\mu)=3$,   and
$-\b_{\mu}(t)=(-b_5,  -t b_5, -t^2 b_5,  -b_4,  -b_1,  -t b_1)$,  
$\mathrm{ev}_\mu (x_1)  = -b_5$, $\mathrm{ev}_\mu (x_2)= -t b_5$, 
$\mathrm{ev}_{\mu} (x_{2} - tx_{1}) = -tb_{5}- t \cdot (-b_{5}) = 0$, etc.
With these notations, we  can prove the following: 
\begin{prop}[Vanishing property]      \label{prop:VanishingProperty(Cohomology)}
 Let $\lambda$, $\mu$ be partitions of length at most $n$ and
set $\hat{\mu}:  =\mu + (1^n) = (\mu_1+1,  \mu_2+1, \ldots,  \mu_n+1)$.
Then  the factorial Hall--Littlewood $P$- and $Q$-polynomials satisfy 
the following vanishing property$:$ 
\begin{enumerate} 
\item 
If $  \mu \not\supset \lambda$,
we have 
\begin{equation*} 
   HQ^{(n)}_{\lambda}(-\b_{\mu}(t),   \underbrace{0,\ldots,0}_{n - \ell (\mu)} ;t|\b)=0 
\quad \text{and}  \quad 
HP^{(n)}_{\lambda}(-\b_{\hat{\mu}}(t);t|\b)=0.
\end{equation*}

\item  
If $\mu=\lambda$,   we have 
\begin{equation*} 
\begin{array}{rl} 
HQ^{(n)}_{\lambda}(-\b_{\lambda}(t),  \underbrace{0,\ldots,0}_{n - \ell (\lambda)} ;t|\b)
&   =\displaystyle
\prod_{q=1}^{\lambda_1}
\prod_{k=1}^{m_q(\lambda)}\left(
\prod_{p=1}^{q}
(-t^{k-1} b_q  +   t^{m_p(\lambda)} b_p)
\right),   \text{and}   \medskip   \\
HP^{(n)}_{\lambda}(-\b_{\hat{\lambda}}(t);t|\b)
&  =
\displaystyle v_{{\lambda>0}}(t)
\prod_{q=2}^{\hat{\lambda}_1}
\prod_{k=1}^{m_q(\hat{\lambda})}
\left (
\prod_{p=1}^{q-1}(-t^{k-1} b_{q} +   t^{m_{p}(\hat{\lambda})} b_{p})
\right ).   \medskip 
\end{array} 
\end{equation*} 
\end{enumerate} 
\end{prop}

\begin{proof}
We only  prove the case of $HP^{(n)}_\lambda (\x_{n}; t|\b)$.  The case 
of $HQ^{(n)}_{\lambda} (\x_{n}; t|\b)$  can be proved similarly.  

(1) As $\lambda\not \subset \mu$, we can find minimal $k$ such that 
$\lambda_k>\mu_k$  ($1\leq k\leq  \ell(\lambda) = r$).
For each choice  $w$ of  $\overline{w}   \in S_n/(S_{1})^{r} \times S_{n-r}$,
we will show the corresponding summand in (\ref{defn:DefinitionFH-LPQ}) vanishes, i.e., 
\begin{equation*} 
\left ( 
w   \cdot \left [  
                        [x_1|\b]^{\lambda_1} \cdots
                        [x_r|\b]^{\lambda_r}
              \prod_{1 \leq i  \leq r, \; i < j \leq n}
               \frac{x_i-t x_j}{x_i-x_j}
             \right ] 
\right )_{\x_{n}  \rightarrow -\b_{\hat{\mu}}(t)}     
=0.
\end{equation*} 

\noindent
For  the permutation $w$,  take minimal $d$ ($1 \leq d\leq k$) such that $w(d)\geq k$. 
Then we divide the discussion into two cases:  \\
\noindent 
Case 1. $w(d)=1$ or  [($w(d)>1$ and $\mu_{w(d)-1}>\mu_{w(d)}$].
In this case, 
\begin{equation*} 
    \big  (  [x_{w(d)}|\b]^{\lambda_d}  \big  )_{x_{w(d)} \to    \mathrm{ev}_{\hat{\mu}}(x_{w(d)})}=0
\end{equation*} 
because $\mathrm{ev}_{\hat{\mu}}(x_{w(d)})=  -b_{\mu_{w(d)}+1}$ and 
$\lambda_d\geq \lambda_k>\mu_k\geq \mu_{w(d)}$. \\
\noindent
Case 2. $w(d)>1$ and $\mu_{w(d)-1}=\mu_{w(d)}$.
In this case,  we claim that 
\begin{equation*} 
      \mathrm{ev}_{\hat{\mu}}   \left   (
                                                     \displaystyle\prod_{1  \leq i  \leq r,\; i< j  \leq n}
                                             \frac{x_{w(i)}-t x_{w(j)}}{x_{w(i)}-x_{w(j)}}
                                        \right  )=   0.
\end{equation*} 
First note that, by the minimality of the choice of $k$,  we have $w(d)>k$.
Let  $p$  ($1 \leq p \leq n$) be an integer such that $w(p)=w(d)-1$.  
Then,  
by the minimality of $d$, we have   $p > k$, and hence $d < p \leq n$.  
Since $\mu_{w(p)} =  \mu_{w(d)}$ and $w(d) =  w(p)+1$,  
we have $\mathrm{ev}_{\hat{\mu}} ( x_{w(d)} ) = t   \cdot  \mathrm{ev}_{\hat{\mu}}  (x_{w(p)} )$.
As $1\leq d  \leq r$, and $d < p \leq n$,   the  factor $\mathrm{ev}_{\hat{\mu}} ( x_{w(d)}-t x_{w(p)} )$
vanishes, and therefore our claim follows.

(2)
When $\mu=\lambda$, we first show that each summand corresponding to  
$\ol{w} \in S_{n}/(S_1)^r \times S_{n-r})$
vanishes  under the evaluation $\mathrm{ev}_{\hat{\lambda}}$,  except for 
$\ol{w} = \ol{e}$ ($e$ the identity element).
In fact, if  $\ol{w}  \neq  \ol{e}$,   we can find minimal $d$ such that $1\leq d  \leq r$ and $w(d) > d$. 
Then, by dividing the argument into two cases Case 1. $\lambda_{w(d) - 1} > \lambda_{w(d)}$, 
 and Case 2. $\lambda_{w(d)-1} = \lambda_{w(d)}$, 
 we can show that the corresponding summand vanishes under the evaluation
 $\mathrm{ev}_{\hat{\lambda}}$. 

For  $w=  e$, we can evaluate the term as follows.
 For each $i$ ($1 \leq i \leq r$),
 we can write $\mathrm{ev}_{\hat{\lambda}}(x_i)  =  t^{k-1} b_q$ 
($k\geq 1$, $q=\lambda_i+1\geq 2$).
 Then, the direct computation yields 
\begin{equation*}  
   \mathrm{ev}_{\hat{\lambda}} \left( [x_i|\b]^{\lambda_i}
 \prod_{j=i+1}^{n} \frac{x_i-t x_j}{x_i-x_j}
 \right)=\frac{1-t^{m_q(\hat{\lambda})-k+1}}{1-t}
 \prod_{p=1}^{q-1} ( 
                                 -t^{k-1} b_q +  t^{m_p (\hat{\lambda})} b_p  
                       ).
 \end{equation*} 
We then  take all the product of these evaluations for $1\leq i \leq r$.   
Since we have $\prod_{q=2}^{\hat{\lambda}_{1}} \prod_{k=1}^{m_{q}(\hat{\lambda})} 
     \frac{1 - t^{m_{q}(\hat{\lambda}) - k + 1}}{1 - t}  = v_{\lambda > 0}(t)$, 
we get the desired formula. 
\end{proof}

More generally, we can prove the vanishing property of the universal factorial 
Hall--Littlewood $P$- and $Q$-functions by the similar way. 
We only exhibit the result.  To state the result, we prepare some notations. 
For a partition $\mu$, we define 
\begin{equation*} 
\ol{\b}_{\mu} [t]
 :=  (\ol{\b}_{\mu_1}^{m_{\mu_1}} [t], 
      \ol{\b}_{\mu_1-1}^{m_{\mu_1-1}} [t], 
       \ldots, \ol{\b}_2^{m_2} [t],  
    \ol{\b}_1^{m_1} [t]),
\end{equation*} 
where   $\ol{\b}_i^k [t] := (\ol{b}_i, [t](\ol{b}_i) ,\ldots,  [t^{k-1} ](\ol{b}_i))$
(we set $\ol{\b}_i^0 [t]  =( \; )$ i.e.,  the empty sequence).

\begin{prop} [Vanishing property]    \label{prop:VanishingProperty(Cobordism)} 
 Let $\lambda$, $\mu$ be partitions of length at most $n$ and
set $\hat{\mu}   =\mu + (1^n) = (\mu_1+1,  \mu_2+1, \ldots,  \mu_n+1)$.
Then  the universal factorial Hall--Littlewood $P$- and $Q$-functions satisfy 
the following vanishing property$:$ 
\begin{enumerate} 
\item
If $ \mu \not \supset \lambda$, we have 
\begin{equation*} 
  HQ^{\L, (n)}_{\lambda}(\ol{\b}_{\mu} [t],  \underbrace{0,\ldots,0}_{n - \ell (\mu)};t|\b)=0 
     \quad    \text{ and }  \quad    
   HP^{\L, (n)}_{\lambda}(\ol{\b}_{\hat{\mu}} [t];t|\b)=0.
\end{equation*}

\item
If $\mu=\lambda$,  we have 
\begin{equation*} 
\begin{array}{rl} 
HQ^{\L, (n)}_{\lambda}(\ol{\b}_{\lambda} [t],  \underbrace{0,\ldots,0}_{n - \ell (\lambda)};t|\b)
&   =\displaystyle
\prod_{q=1}^{\lambda_1}
\prod_{k=1}^{m_q(\lambda)}\left(
\prod_{p=1}^{q}
([t^{k-1}](\ol{b}_q)  +_\L [t^{m_p(\lambda)}] (b_p))
\right),
 \text{and}   \medskip \\
HP^{\L, (n)}_{\lambda}(\ol{\b}_{\hat{\lambda}} [t] ;t|\b)
&  =
\displaystyle v_{{\lambda>0}}(t)
\prod_{q=2}^{\hat{\lambda}_1}
\prod_{k=1}^{m_q(\hat{\lambda})}
\left (
\prod_{p=1}^{q-1}([t^{k-1}] (\ol{b}_{q}) +_\L [t^{m_{p}(\hat{\lambda})}] (b_{p}))
\right ).
\end{array} 
\end{equation*} 
\end{enumerate} 
\end{prop}

\subsection{Pieri-type  formula and Hook formula}
The vanishing property established in the previous section is so useful that 
one  can derive  several interesting results of factorial Hall--Littlewood polynomials 
from this. 
Denote by  $\Lambda (\x_{n}) = \Z[x_{1},  \ldots, x_{n}]^{S_{n}}$  the ring of 
symmetric polynomials of $n$ variables, and $\mathcal{P}_{n}$  
the set of partitions of length $\leq n$.  
Then,  it is known that the usual Hall--Littlewood $P$-polynomials 
$P_{\lambda}(\x_{n}; t)$ ($\lambda \in \mathcal{P}_{n}$) form 
a $\Z[t]$-basis of $\Lambda (\x_{n})[t] \cong  \Z[t] \otimes_{\Z}   \Lambda (\x_{n})$ 
(cf. Macdonald \cite[III, (2.7)]{Macdonald1995}. 
Therefore there exist polynomials $c_{\lambda, \mu}^{\nu} (t) 
 = c_{\lambda, \mu}^{\nu, (n)}(t)  \in \Z[t]$ such 
that 
\begin{equation*} 
    P_{\lambda}(\x_{n}; t) P_{\mu}(\x_{n}; t) 
 =  \sum_{\nu} c_{\lambda, \mu}^{\nu} (t)  P_{\nu}(\x_{n}; t)
  \quad (\lambda, \mu,  \nu \in \mathcal{P}_{n}). 
\end{equation*}  
It is known that (see Macdonald \cite[III, (5.7)]{Macdonald1995}) 
the following Pieri-type formula holds: 
\begin{equation}   \label{eqn:Pieri-typeFormulaP_lambda(x_n;t)}
P_{(1)}(\bm{x}_{n}; t) P_{\mu}(\bm{x}_{n}; t)
   =  \sum_{\mu \subset \nu, \; |\nu/\mu| = 1} 
   \alpha_{\nu/\mu} (t)    P_{\nu}(\bm{x}_{n}; t), 
  \end{equation}  
where     the polynomial  $\alpha_{\nu/\mu} (t)  = \alpha_{\nu/\mu}^{(n)}(t)$ is given by
 $\dfrac{1 - t^{m_{j}(\nu)}}{1 - t}$
 if $\nu/\mu$ has a box in $j$th column.  
As for the factorial version of Macdonald's Hall--Littlewood $P$-polynomials 
$P_{\lambda}(\bm{x}_{n}; t|\bm{b})$ (see (\ref{eqn:DefinitionP_lambda(x_n;t|b)})), 
one can consider a similar problem: 
First we see that factorial Hall--Littlewood $P$-Polynomials 
$P_{\lambda}(\bm{x}_{n}; t|\bm{b})$ ($\lambda \in \mathcal{P}_{n}$) form 
a $\Z[t] \otimes_{\Z}  \Z[\bm{b}]$-basis of $\Lambda (\x_{n}|\bm{b}) [t] := 
\Z[t] \otimes_{\Z} \Z[\bm{b}] \otimes_{\Z} \Lambda (\bm{x}_{n})$, 
where $\Z[\bm{b}] = \Z[b_{1}, b_{2},  \ldots] $ is a polynomial ring 
of indeterminates $\bm{b}= (b_{1}, b_{2}, \ldots)$.  
Therefore there exists polynomials $c_{\lambda, \mu}^{\nu}  (t|\bm{b}) 
 =  c_{\lambda, \mu}^{\nu, (n)}(t|\bm{b})   \in
\Z[t] \otimes \Z[\bm{b}]$ such that 
\begin{equation}   \label{eqn:c_lambdamu^nu(t|b)}  
    P_{\lambda}(\x_{n}; t|\bm{b}) P_{\mu}(\x_{n}; t|\bm{b}) 
 =  \sum_{\nu} c_{\lambda, \mu}^{\nu} (t|\bm{b})  P_{\nu}(\x_{n}; t|\bm{b})
  \quad (\lambda, \mu,  \nu \in \mathcal{P}_{n}). 
\end{equation}  
By definition,  the ``structure constants'' $c_{\lambda, \mu}^{\nu}(t|\bm{b})$
 is a homogeneous polynomial of 
degree $|\lambda| + |\mu| - |\nu|$ in the indeterminates $\bm{b} = (b_{1}, b_{2}, \ldots)$ 
with coefficients in $\Z[t]$.  Comparing the highest homogeneous components 
in $\bm{x}_{n} = (x_{1}, \ldots, x_{n})$ on  both sides of (\ref{eqn:c_lambdamu^nu(t|b)}), we see that 
\begin{equation*} 
      c_{\lambda, \mu}^{\nu}(t|\bm{b}) 
    = \left \{ 
                  \begin{array}{lll} 
                        &   c_{\lambda, \mu}^{\nu} (t) &  \quad  \text{if}  \;  |\lambda| + |\mu| =  |\nu|, \\
                          &  0                                 &  \quad   \text{if} \;  |\lambda| + |\mu| <    |\nu|.  
                 \end{array} 
         \right. 
\end{equation*} 
From  the commutativity of the product in the left hand side of 
(\ref{eqn:c_lambdamu^nu(t|b)}),
 the symmetry $c_{\lambda, \mu}^{\nu}(t|\b)
=  c_{\mu, \lambda}^{\nu} (t|\b)$ holds obviously.  
Furthermore, using the vanishing property, Proposition \ref{prop:VanishingProperty(Cohomology)}, 
and the totally same argument as in Molev--Sagan \cite[p.4434]{Molev-Sagan1999}, 
we see that $c_{\lambda, \mu}^{\nu} (t|\bm{b})$ is zero unless $\lambda \subset \nu$ and 
$\mu \subset \nu$.

Now  we consider the case where $\lambda = (1)$ in (\ref{eqn:c_lambdamu^nu(t|b)}). 
  Then,   by the known properties of the structure constants, we only need to consider 
those $\nu$ with $\mu \subset \nu$ and $|\nu| \leq  |\mu| + 1$, 
Thus (\ref{eqn:c_lambdamu^nu(t|b)}) takes the following form: 
\begin{equation*} 
   P_{(1)}(\x_{n}; t|\bm{b}) P_{\mu}(\x_{n}; t|\bm{b}) 
   =   c_{(1), \mu}^{\mu}(t|\bm{b}) P_{\mu}(\bm{x}_{n};  t | \bm{b})  + 
          \sum_{\mu \subset \nu, \; |\nu/\mu| = 1} c_{(1), \mu}^{\nu} (t|\bm{b}) P_{\nu}(\bm{x}_{n}; t |\bm{b}). 
\end{equation*} 
Setting $\bm{x}_{n}  = - \bm{b}_{\hat{\mu}} (t)$ and using the vanishing property, 
we see that $c_{(1), \mu}^{\mu} (t|\bm{b}) =  P_{(1)}(-\bm{b}_{\hat{\mu}}(t); t |\bm{b})$. 
On the other hand, by the degree reason, we have $c_{(1), \mu}^{\nu} (t|\bm{b}) = 
c_{(1), \mu}^{\nu} (t) = \alpha_{\nu/\mu}(t)$ when $\mu \subset \nu$ and $|\nu/\mu|  = 1$. 
Therefore we obtain the following  formula: 
\begin{prop}[Pieri-type formula for factorial Hall--Littlewood $P$-polynomials]  
\label{prop:Pieri-typeFormulaFH-LP} 
\begin{equation*} 
   P_{(1)}(\x_{n}; t|\bm{b}) P_{\mu}(\x_{n}; t|\bm{b}) 
    =  P_{(1)}(-\bm{b}_{\hat{\mu}}(t); t |\bm{b})  P_{\mu}(\bm{x}_{n};  t | \bm{b})  + 
          \sum_{\mu \subset \nu, \; |\nu/\mu| = 1} \alpha_{\nu/\mu}  P_{\nu}(\bm{x}_{n}; t | \bm{b}). 
\end{equation*} 
\end{prop} 

Using Proposition \ref{prop:Pieri-typeFormulaFH-LP},  we  can derive 
a  generalization of   the  so-called {\it hook} ({\it length})  {\it formula}. 
We argue as follows (the following argument is essentially the same as 
that given in Molev--Sagan \cite[Proposition 3.2]{Molev-Sagan1999}
for factorial Schur polynomials, 
although they did not mention the relation to the hook formula.  
For this type of argument, see also Naruse--Okada \cite[Lemma 4.5]{Naruse-Okada2019}). 
For simplicity,  we shall 
use the abbreviated notation $P_{\lambda}$,  
$c_{\lambda, \mu}^{\nu}$, and $\alpha_{\lambda/\mu}$  for  
 $P_{\lambda}(\x_{n} ; t|\b)$, 
$c_{\lambda, \mu}^{\nu, (n)}(t|\bm{b})$, and $\alpha_{\lambda, \mu}^{(n)}(t)$
respectively  in the following.  Then our hook  formula is stated as follows: 
\begin{prop} [Hook formula for factorial Hall--Littlewood $P$-polynomials]    \label{prop:OurHookFormula}  
Let $\mu$  be a partition of length  $\ell (\mu) \leq n$ and size $|\mu| = k$,  a positive integer.   
Then we have the following formula$:$ 
\begin{equation}   \label{eqn:OurHookFormula} 
     \sum_{\mu = \mu^{(0)} \supsetneq \mu^{(1)}  \supsetneq \mu^{(2)} \supsetneq \cdots \supsetneq \mu^{(k)} = \emptyset} 
 \dfrac{\alpha_{\mu^{(k-1)}/\mu^{(k)}}
           } 
           {c_{(1), \mu}^{\mu} - c_{(1), \mu^{(k)}}^{\mu^{(k)}} 
            } 
   \cdot \cdots \cdot
\dfrac{\alpha_{\mu^{(1)}/\mu^{(2)}}
           } 
           {c_{(1), \mu}^{\mu} - c_{(1), \mu^{(2)}}^{\mu^{(2)}} 
            }
    \cdot 
  \dfrac{\alpha_{\mu^{(0)}/\mu^{(1)}}
           } 
           {c_{(1), \mu}^{\mu} - c_{(1), \mu^{(1)}}^{\mu^{(1)}} 
            } 
   =  \dfrac{1}{  P_{\mu}(-\bm{b}_{\hat{\mu}} (t); t|\bm{b})  }.    
\end{equation} 
\end{prop}

\begin{proof} 
We consider the associativity of  the product 
\begin{equation*} 
    (P_{(1)}  P_{\lambda}) P_{\mu} = P_{(1)} (P_{\lambda} P_{\mu}), 
\end{equation*} 
and take the coefficient of $P_{\mu}$ on both sides. 
Using the fact that 
$c_{\alpha, \beta}^{\gamma}$ is zero unless $\alpha \subset \gamma$ and $\beta \subset \gamma$, 
and   Proposition \ref{prop:Pieri-typeFormulaFH-LP},   we have
\begin{equation*} 
    c_{(1), \lambda}^{\lambda} c_{\lambda, \mu}^{\mu}  +  
     \sum_{
                     \mu \supset \nu \supsetneq \lambda, \;  |\nu/\lambda| = 1
       }   \alpha_{\nu/\lambda}  \,   c_{\nu, \mu}^{\mu}  
       =   c_{(1), \mu}^{\mu}  c_{\lambda, \mu}^{\mu}, 
\end{equation*}  
and therefore we have 
\begin{equation*} 
     (c_{(1), \mu}^{\mu} - c_{(1), \lambda}^{\lambda}) c_{\lambda, \mu}^{\mu} 
  =   \sum_{
                     \mu \supset \nu \supsetneq \lambda, \;  |\nu/\lambda| = 1
       }   \alpha_{\nu/\lambda}  \,  c_{\nu, \mu}^{\mu}. 
\end{equation*} 
By  definition and Example \ref{ex:FH-LPP},  we know that 
$P_{(1)}(\bm{x}_{n}; t|\bm{b})  =   x_{1}+ \cdots + x_{n}  +  \frac{1-t^{n}}{1-t}b_{1}$. 
Therefore, if $\mu \supsetneq \lambda$, we see that  
$c_{(1), \mu}^{\mu} - c_{(1), \lambda}^{\lambda}
 = P_{(1)}(-\bm{b}_{\hat{\mu}}(t); t|\bm{b}) - P_{(1)}(-\bm{b}_{\hat{\lambda}}(t); t|\bm{b}) 
\neq 0$.   Thus we have the following recurrence formula:  
\begin{equation*} 
    c_{\lambda,   \mu}^{\mu}  
     =   \sum_{
                     \mu \supset \nu \supsetneq \lambda, \;  |\nu/\lambda| = 1
       }    \dfrac{ \alpha_{\nu/\lambda} } 
                {  c_{(1), \mu}^{\mu} - c_{(1), \lambda}^{\lambda} }   c_{\nu, \mu}^{\mu}. 
\end{equation*} 
Using this recurrence formula repeatedly, we obtain 
\begin{equation*} 
c_{\emptyset, \mu}^{\mu}
   =   \sum_{\mu = \mu^{(0)} \supsetneq \mu^{(1)}  \supsetneq \mu^{(2)} \supsetneq \cdots \supsetneq \mu^{(k)} = \emptyset} 
 \dfrac{\alpha_{\mu^{(k-1)}/\mu^{(k)}}
           } 
           {c_{(1), \mu}^{\mu} - c_{(1), \mu^{(k)}}^{\mu^{(k)}} 
            } 
   \cdot \cdots \cdot
\dfrac{\alpha_{\mu^{(1)}/\mu^{(2)}}
           } 
           {c_{(1), \mu}^{\mu} - c_{(1), \mu^{(2)}}^{\mu^{(2)}} 
            }
    \cdot 
  \dfrac{\alpha_{\mu^{(0)}/\mu^{(1)}}
           } 
           {c_{(1), \mu}^{\mu} - c_{(1), \mu^{(1)}}^{\mu^{(1)}} 
             } c_{\mu, \mu}^{\mu}.    
\end{equation*} 
The fact that $c_{\emptyset, \mu}^{\mu} = 1$ is obvious from the definition of structure constants. 
The value of $c_{\mu, \mu}^{\mu}$ equals to $P_{\mu}(-\bm{b}_{\hat{\mu}}(t); t|\bm{b})$ 
by virtue of the vanishing property, Proposition \ref{prop:VanishingProperty(Cohomology)}. 
Therefore, we have the desired equation. 
\end{proof}

As mentioned  before the proposition,
one can obtain  a similar hook formula  by \cite[Proposition 3.2]{Molev-Sagan1999}.     
More concretely, 
under their notation, one has the following formula$:$  
\begin{equation}   \label{eqn:Molev-SaganColoredHookFormula} 
  \sum_{\emptyset = \rho^{(0)}  \rightarrow \rho^{(1)} \rightarrow \cdots \rightarrow \rho^{(l-1)}  
  \rightarrow     \rho^{(l)} = \nu}  
\dfrac{1}{  (|a_{\nu}|  - |a_{\rho^{(0)}}|)
                      \cdots 
               (|a_{\nu}|  -  |a_{\rho^{(l-1)}}|) 
    }   =  \dfrac{1}{s_{\nu}  (a_{\nu}|a)}. 
\end{equation} 
We remark that this formula can be interpreted as  a special case  
of   Nakada's  {\it colored hook formula} 
 (\cite[Corollary 7.2]{Nakada2008}), 
which is a generalization of the famous hook  formula 
due to Frame--Robinson--Thrall \cite{Frame-Robinson-Thrall1954}. 
As an  example, let us take $\nu = (2, 2)$ and $n = 2$, the number of variables.  
Then the above formula leads to 
\begin{equation*} 
\begin{array}{ll} 
   &   \dfrac{1}{(a_{3} - a_{2}) (a_{3} - a_{1}) (a_{4} - a_{1}) (a_{4} + a_{3} - a_{2} - a_{1})}  \medskip \\
     & \hspace{6.5cm}   + \;  \dfrac{1}{(a_{3} - a_{2}) (a_{4} - a_{2}) (a_{4} - a_{1}) (a_{4} + a_{3} - a_{2} - a_{1})}  \medskip \\
   &  =  \;  \dfrac{1}{(a_{3} - a_{2})(a_{3} - a_{1})(a_{4} - a_{2}) (a_{4} - a_{1})}. 
\end{array}  
\end{equation*} 
Now consider the simple system $\{ \alpha_{1},  \alpha_{2}, \alpha_{3} \}$ of the root system of type $A_{3}$. 
If one represent the simple root $\alpha_{i}$ as $a_{i} - a_{i + 1}$ for $i = 1, 2, 3$, 
then the above identity becomes 
\begin{equation}
\begin{array}{lll} 
&  \dfrac{1}{\alpha_{2} (\alpha_{1} + \alpha_{2})(\alpha_{1} + \alpha_{2} + \alpha_{3})
       (\alpha_{1} + 2\alpha_{2} + \alpha_{3}) }  
      +     \dfrac{1}{\alpha_{2} (\alpha_{2} + \alpha_{3})(\alpha_{1} + \alpha_{2} + \alpha_{3})
       (\alpha_{1} + 2\alpha_{2} + \alpha_{3}) }    \medskip \\
 & =  \dfrac{1}{\alpha_{2} (\alpha_{1} + \alpha_{2}) (\alpha_{2} + \alpha_{3}) (\alpha_{1} + \alpha_{2} + \alpha_{3})},  
\end{array} 
\end{equation} 
which agrees with the example given in \cite[p.1088]{Nakada2008}. 
When we specialize $t$ to be $0$, our factorial Hall--Littlewood 
$P$-polynomial $HP_{\lambda}(\bm{x}_{n}; 0 |\bm{b}) = P_{\lambda}(\bm{x}_{n}; 0 |\bm{b})$
does not coincide  with the factorial Schur polynomial 
$s_{\lambda}(\bm{x}_{n}|\bm{b})$.\footnote{
In the definition of the factorial Schur polynomial $s_{\lambda}(x|a)$ given 
by Molev--Sagan \cite[\S 2, (3)]{Molev-Sagan1999}, we replaced 
a doubly-infinite variable sequence $a = (a_{i})$, $i \in \Z$, 
by $\bm{b} = (b_{1}, b_{2},  \ldots)$.       
}
Thus $t = 0$ specialization of our hook formula  (\ref{eqn:OurHookFormula}) yields  
{\it another}  colored hook formula (see the example below).  
 It is well-known that the classical hook formula and its {\it shifted} analogue 
have geometric  background known as Schubert calculus, 
and are  closely related to combinatorics
 of Grassmannians, root systems, and Weyl groups (see e.g., Hiller \cite{Hiller1982}). 
In our forthcoming paper (\cite{Nakagawa-Naruse2022}), 
we shall discuss  geometric or topological background of 
our hook formula,  in relation to complex reflection groups 
$G(e, 1, n)$  and  $G(e, e, n)$ (for root systems of these groups, see 
Bremke--Malle \cite{Bremke-Malle1997} \cite{Bremke-Malle1998}).

\begin{ex}   \label{ex:OurHookFormula} 
For the partition $\mu = (2, 2)$, the explicit form of our hook length formula is 
given as follows$:$ 
First note that there exist ``two paths'' from $\mu = (2, 2)$ to $\emptyset = (\;  )$. 
Namely, 
\begin{equation*} 
 \mu = (2, 2)  \supsetneq (2, 1) \supsetneq (2) \supsetneq (1) \supsetneq (\;  ) \quad 
  \text{and}  \quad   
     \mu = (2, 2)  \supsetneq (2, 1) \supsetneq (1, 1) \supsetneq (1)   \supsetneq (\; ). 
\end{equation*} 
From the fact that  $c_{(1), \nu}^{\nu}  = c_{(1), \nu}^{\nu, (n)}(t|\bm{b}) 
=  P_{(1)}^{(n)}(-\bm{b}_{\hat{\nu}} (t); t|\bm{b})$, 
we get the following result directly$:$ 
\begin{equation*} 
\begin{array}{lll} 
     c_{(1), (\; )}^{(\; )} = 0, \medskip \\
     c_{(1), (1)}^{(1)} =  -b_{2} + t^{n-1} b_{1}, \medskip \\
     c_{(1), (1, 1)}^{(1, 1)} = (1 + t)(-b_{2} + t^{n-2}b_{1}), \medskip \\
     c_{(1), (2)}^{(2)} =  -b_{3} + t^{n-1}b_{1}, \medskip \\
     c_{(1), (2, 1)}^{(2, 1)} = - b_{3} - b_{2}  +  (1 + t)t^{n-2}b_{1}, \medskip  \\
     c_{(1), (2, 2)}^{(2, 2)}  =  (1 + t)(-b_{3} + t^{n-2}b_{1}). \medskip  
\end{array} 
\end{equation*} 
 Similarly, $\alpha_{\nu/\lambda} =  \alpha_{\nu/\lambda}^{(n)}(t)$ can be computed directly 
from the definition, and we get 
\begin{equation*} 
    \alpha_{(2, 2)/(2, 1)}  = 1 + t, \; 
   \alpha_{(2, 1)/(2)} =  1, \; 
    \alpha_{(2, 1)/(1, 1)} = 1, \; 
      \alpha_{(2)/(1)} = 1, \; 
  \alpha_{(1, 1)/(1)} = 1 + t, \; 
   \alpha_{(1)/(\; )}  =  1. 
\end{equation*} 
By Proposition $\ref{prop:VanishingProperty(Cohomology)}$,  we have, for $\mu = (2, 2)$, 
\begin{equation*} 
    P_{\mu}^{(n)} (-\bm{b}_{\hat{\mu}}(t); t|\bm{b})  
  =  (-b_{3} + t^{n-2}b_{1}) (-tb_{3}  + t^{n-2}b_{1}) (-b_{3} + b_{2}) (-tb_{3} + b_{2}). 
\end{equation*} 
Therefore our hook  formula gives the following identity$:$ 
\begin{equation}  
\begin{array}{lll} 
& \dfrac{1+t}{-t b_{3}  + b_{2}}  \cdot 
 \dfrac{1}{-t b_3 + t^{n-2} b_1}  \cdot
\dfrac{1}{-b_3 - t b_{3} +  b_{2} +  t^{n-2}b_1} \cdot
\dfrac{1}{-b_3 - t b_3 +  t^{n-2} b_1 + t^{n-1}b_1}  \medskip 
\\
& +
\dfrac{1+t}{-t b_3 + b_2} \cdot
\dfrac{1}{-b_3 - t b_3 +  b_2 + t b_2} \cdot
\dfrac{1+t}{-b_3 - t b_3 + b_2 + t^{n-2}b_1} \cdot
\dfrac{1}{-b_3 - t b_3 +  t^{n-2} b_1 + t^{n-1}b_1}  \medskip \\
& =\dfrac{1}{(- b_3 + t^{n-2} b_1)(-t b_3 + t^{n-2} b_1)(-b_3 + b_2)(-t b_3+ b_2)}.  \medskip 
\end{array} 
\end{equation} 
\end{ex}

\section{Generating functions for the universal factorial Hall--Littlewood $P$- and $Q$-functions}    \label{sec:GFUFH-LPQ}  
In this section,  by utilizing  a Gysin   formula
in complex cobordism, Proposition \ref{prop:FundamentalFormula(TypeA)(ComplexCobordism)}, 
 we  shall  derive  the generating functions for  the universal factorial   
Hall--Littlewood $P$- and $Q$-functions.

\subsection{Generating function for $HP^{\L}_{\lambda}(\x_{n}; t|\b)$ }      \label{subsec:GFUFH-LP}
Basic idea is  to apply the fundamental formula (\ref{eqn:FundamentalFormula(TypeA)(ComplexCobordism)})
repeatedly  to   the characterization (\ref{eqn:CharacterizationUFH-LP}) to 
obtain the generating function. 
Here we remark that the formula (\ref{eqn:FundamentalFormula(TypeA)(ComplexCobordism)}) 
still holds for a formal power series $f(u)  \in MU^{*}(X)[[u]]$ as well, 
and we shall use such an extended form of (\ref{eqn:FundamentalFormula(TypeA)(ComplexCobordism)}). 
 However, we will be confronted with some difficulty when we 
apply the formula to (\ref{eqn:CharacterizationUFH-LP}).     
 In order to clarify the difficulty,   let us consider 
the simplest case  $\lambda = (\lambda_{1})$ with $\lambda_{1} \geq 1$ 
 (and hence $r = 1$)   
of (\ref{eqn:CharacterizationUFH-LP}). 
  We wish to push-forward the expression  
$[x_{1}|\b]_{\L}^{\lambda_{1}}  \prod_{j=2}^{n}  (x_{1}  +_{\L}  [t] (\ox_{j}))$ 
via the Gysin map $\pi^{1}_{*}:  MU^{*}(G^{1}(E))  \longrightarrow MU^{*}(X)$. 
Naively, setting  
\begin{equation*} 
  f(u) :=  [u|\b]_{\L}^{\lambda_{1}}  \cdot   \prod_{j=2}^{n}  (u +_{\L}  [t] (\ox_{j})),
\end{equation*}  
we wish to compute $\pi^{1}_{*}(f(x_{1}))$.  However, one cannot regard $f(u)$ 
as an element of $MU^{*}(X)[[u]]$ as it is.  
Therefore we consider the following  expression   instead: 
\begin{equation*} 
    f_{1}(u) :=  \dfrac{[u|\b]_{\L}^{\lambda_{1}}} {u +_{\L}  [t] (\overline{u})} 
            \cdot  \displaystyle{\prod_{j=1}^{n}}  (u +_{\L}  [t] (\ox_{j})).  
\end{equation*} 
Since  symmetric functions in $x_{1}, \ldots, x_{n}$
 can be  regarded as elements of $MU^{*}(X)$ ($x_{1}, \ldots, x_{n}$
are the Chern roots of $E$),  the coefficients of $f_{1}(u)$ with respect to $u$ 
are actually in $MU^{*}(X)$.  
Moreover, we have $f(x_{1}) = f_{1}(x_{1})$ obviously. 
However,    it is not a formal power series in $u$ 
because  of the constant term $b_{1}b_{2} \cdots b_{\lambda_{1}}$ in the 
numerator,   and 
therefore the formula (\ref{eqn:FundamentalFormula(TypeA)(ComplexCobordism)}) 
does not apply directly.   
We further modify $f_{1}(u)$, and consider the following expression: 
\begin{equation}   \label{eqn:[x_1|b]^lambda_1(2)} 
  f_{2}(u) :=  
    \dfrac{[u|\b]_{\L}^{\lambda_{1}}} {u +_{\L}  [t](\overline{u})} 
         \left \{   \prod_{j=1}^{n} (u +_{\L}  [t] (\ox_{j}))  
                 -  \prod_{j=1}^{n}  [t] (u +_{\L}  \ox_{j}) 
         \right \}.   
\end{equation} 
The effect of subtracting the term $\prod_{j=1}^{n} [t] (u +_{\L}  \ox_{j})$
(hereafter  we call it the ``correction term'') is two-fold: 
Firstly,  the expression  $\prod_{j=1}^{n} (u +_{\L} [t] (\ox_{j})) 
 -   \prod_{j=1}^{n} [t] (u +_{\L}  \ox_{j})$
is divisible by $u$,   and   therefore $f_{2}(u)$ becomes  indeed a formal power series
 in $u$ with coefficients in $MU^{*}(X)$. Secondly, when we  substitute  $x_{1}$ for $u$,  
we have $f(x_{1}) = f_{2}(x_{1})$ by the obvious identity $\prod_{j=1}^{n} [t] (x_{1} +_{\L}  \ox_{j}) = 0$.  
Therefore  the fundamental Gysin formula 
(\ref{eqn:FundamentalFormula(TypeA)(ComplexCobordism)}) does  apply to $f_{2}(u)$,    
and the result is given as follows: 
\begin{equation*} 
\begin{array}{lll} 
   &  HP^{\L}_{(\lambda_{1})} (\x_{n}; t|\b)  
   =   \pi^{1}_{*} (f_{2}(x_{1})) 
    =  [u^{n-1}]   (f_{2}(u) \times \mathscr{S}^{\L} (E; 1/u))  \medskip \\
  & =  [u^{n-1}] 
          \left [ 
                   \dfrac{[u|\b]_{\L}^{\lambda_{1}}} {u +_{\L}  [t] (\ou_{1})} 
         \left \{   \displaystyle{\prod_{j=1}^{n}} (u +_{\L}  [t] (\ox_{j}))  
                -    \prod_{j=1}^{n}  [t] (u +_{\L}  \ox_{j}) 
          \right \}   
          \times  \mathscr{S}^{\L} (E; 1/u)   
          \right ]     \medskip  \\
 & =  [u^{-\lambda_{1}}] 
         \left [ 
                     \dfrac{1}{\mathscr{P}^{\L} (u)}  \dfrac{u} {u +_{\L}  [t](\ou)} 
                      \left \{ 
                                  \displaystyle{\prod_{j=1}^{n}}   
                                                    \dfrac{u +_{\L}  [t] (\ox_{j})} {u +_{\L} \ox_{j}}     
                -    \prod_{j=1}^{n}   \dfrac{[t] (u +_{\L}  \ox_{j})} {u +_{\L} \ox_{j}}       
                       \right \}  
                            \times  \displaystyle{\prod_{j=1}^{\lambda_{1}}} \dfrac{u +_{\L} b_{j}} {u}  
         \right ].   
\end{array} 
\end{equation*}
\begin{ex}  
As a special case of  the above formula,  the ordinary factorial Hall-Littlewood $P$-polynomial 
corresponding to the one-row $(\lambda_{1})$ is given by 
\begin{equation*} 
   HP_{(\lambda_{1})}(\x_{n}; t |\b)  
  =  [u^{-\lambda_{1}}]  
        \left [ 
                                  \dfrac{1}{1-t}  
                                      \left ( 
                                                  \prod_{j=1}^{n}  \dfrac{u - tx_{j}} {u - x_{j}}  - t^{n} 
                                     \right ) \times  \prod_{j=1}^{\lambda_{1}} \dfrac{u + b_{j}} {u} 
      \right ]. 
\end{equation*} 
In particular,  we have 
\begin{equation*} 
\begin{array}{llll} 
       HP_{(1)}(\x_{n}; t |\b) 
       &  =  [u^{-1}]  
     \left [ 
                                  \dfrac{1}{1-t}  
                                      \left ( 
                                                  \displaystyle{\prod_{j=1}^{n}}  \dfrac{u - tx_{j}} {u - x_{j}}  - t^{n} 
                                     \right ) \times    \dfrac{u + b_{1}} {u} 
      \right ]   \medskip \\
      &  =  \dfrac{1}{1-t} q_{1}(\x_{n} ; t)  +  \dfrac{1-t^{n}}{1-t} b_{1} \medskip \\
      &  =   x_{1} +  x_{2} + \cdots + x_{n}  +  (1 + t + t^{2} + \cdots + t^{n-1}) b_{1}.  
\end{array}  
\end{equation*}  
Here $q_{r}(\x_{n};t) \; (r = 0, 1, 2, \ldots)$ are given by the following generating functions$:$ 
\begin{equation*} 
           \prod_{j=1}^{n}  \left.  \dfrac{z - tx_{j}}{z - x_{j}} \right |_{z = u^{-1}} 
          =  \prod_{j=1}^{n} \dfrac{1 - tx_{j}u}{1 - x_{j}u}  = \sum_{r=0}^{\infty} q_{r}(\x_{n}; t) u^{r}. 
\end{equation*} 
\end{ex}  

For a general sequence of positive integers 
 $\lambda = (\lambda_{1}, \ldots, \lambda_{r})$ with $r \leq n$, 
we need to compute the push-forward image of 
 $[\x|\b]_{\L}^{\lambda}  \prod_{i=1}^{r} \prod_{j=i + 1}^{n} (x_{i} +_{\L} [t] (\ox_{j}))$ 
under the Gysin map $(\pi^{r, r-1, \ldots, 1})_{*}:  MU^{*}(\F \ell^{r, \ldots, 1}(E))   \longrightarrow MU^{*}(X)$.  
The  image of   $(\pi^{r, r-1, \ldots, 1})_{*}$ can be computed by applying $\pi^{r}_{*},  \pi^{r-1}_{*}, \ldots, \pi^{1}_{*}$ 
successively.  In each step,  we  use the  modification 
 such as  (\ref{eqn:[x_1|b]^lambda_1(2)}), i.e.,  subtracting  the ``correction term''. 
This technique enables us to apply the fundamental Gysin formula
 (\ref{eqn:FundamentalFormula(TypeA)(ComplexCobordism)}), 
 and  we are able to show the following result:   
\begin{lem}    \label{lem:GFUFH-LP}  
For  a sequence of positive integers  
  $\lambda = (\lambda_{1}, \ldots, \lambda_{r})$  with $r \leq n$, 
we have the following formula$:$ 
\begin{equation}    \label{eqn:GFUFH-LPI}  
\begin{array}{lll} 
 &  (\pi^{r, r-1, \ldots, 1})_{*}  \left ([\x|\b]_{\L}^{\lambda}  \displaystyle{\prod_{i=1}^{r}} 
                                                   \prod_{j=i + 1}^{n} (x_{i} +_{\L} [t] (\ox_{j}))
                                            \right  ) 
    =   \left [  \displaystyle{\prod_{i=1}^{r}}   u_{i}^{-\lambda_{i}} \right ]  \medskip \\
              &  \left (    \displaystyle{\prod_{i=1}^{r}}  \dfrac{u_{i}} {u_{i} +_{\L} [t] (\ou_{i})} 
                  \cdot  \dfrac{1}{\mathscr{P}^{\L} (u_{i})}    
                  \times   
                  \left \{ 
                            \displaystyle{\prod_{j=1}^{n}}   \dfrac{u_{i} +_{\L}  [t] (\ox_{j})} 
                                                                        {u_{i} +_{\L}  \ox_{j}} 
                                                 -       \prod_{j  = 1}^{i-1}
                                                                \dfrac{u_{i} +_{\L}  [t] (\ou_{j})} 
                                                                     {[t] (u_{i} +_{\L} \ou_{j})}                                          
                                                          \prod_{j=1}^{n} 
                                                                 \dfrac{[t] (u_{i} +_{\L} \ox_{j})} 
                                                                           {u_{i} +_{\L} \ox_{j}} 
                   \right \}    \right.   \medskip \\
     & \left.   \hspace{7cm} 
                 \times  \displaystyle{\prod_{1 \leq i < j \leq r}} 
                                      \dfrac{u_{j} +_{\L}  \ou_{i}}  {u_{j} +_{\L}  [t]  (\ou_{i})}  
                 \times 
        \displaystyle{\prod_{i=1}^{r}} \prod_{j=1}^{\lambda_{i}} \dfrac{u_{i} +_{\L}  b_{j}} {u_{i}}  \right ). \medskip 
\end{array} 
\end{equation} 
\end{lem} 
\begin{proof} 
Let us  compute the 
push-forward image of $[\x|\b]_{\L}^{\lambda}  \prod_{i=1}^{r} \prod_{j=i + 1}^{n} (x_{i} +_{\L} [t] (\ox_{j}))$ 
under the Gysin map $(\pi^{r, r-1, \ldots, 1})_{*}  =  \pi^{1}_{*} \circ \cdots \circ \pi^{r -1}_{*} \circ \pi^{r}_{*}$. 
As we explained above,   
we carry out the computation inductively. 
For $a$ ($= 1, 2, \ldots, r-1$),  we assume the following result:  
\begin{equation}   \label{eqn:InductiveStepMU(G^1(U_n-r+a+1))}  
\begin{array}{lll} 
 & (\pi^{r - a + 1} \circ \cdots  \circ \pi^{r-1} \circ  \pi^{r})_{*}  
   \left   ( [\bm{x}|\b]_{\L}^{\lambda}  
                   \displaystyle{\prod_{i=1}^{r}} \prod_{j=i + 1}^{n} (x_{i} +_{\L} [t] (\ox_{j}))   
          \right )   \medskip \\
 & =    [u_{r - a + 1}^{n-1} \cdots u_{r-1}^{n-1} u_{r}^{n-1}]   
       \left  (     
         \displaystyle{\prod_{i=1}^{r-a}}    [x_{i}|\b]_{\L}^{\lambda_{i}}
              \prod_{j = i  +1}^{n}  (x_{i}  +_{\L}  [t] (\ox_{j})) 
         \right.       \medskip \\
   &     \left.    
             \times    \displaystyle{\prod_{i=r-a + 1}^{r}} 
                          \dfrac{[u_{i}|\b]_{\L}^{\lambda_{i}}} {u_{i} +_{\L} [t] (\ou_{i})} 
                       \left \{ 
                                \prod_{j=r-a + 1}^{n} (u_{i}  +_{\L}  [t] (\ox_{j})) 
                                    -   \prod_{j= r - a + 1}^{i - 1} 
                                                      \dfrac{u_{i} +_{\L}  [t] (\ou_{j})} 
                                                               {[t] (u_{i} +_{\L}  \ou_{j})}  
                           \prod_{j= r-a + 1}^{n} [t]  (u_{i} +_{\L}  \ox_{j})    \right \}      \right.    \medskip \\
   &  \left.    \times    \displaystyle{\prod_{i=r-a + 1}^{r}}  \prod_{j=1}^{r-a}  (u_{i}  +_{\L} \ox_{j})    \times  
        \prod_{r - a + 1 \leq i < j \leq r}  \dfrac{ u_{j} +_{\L}  \ou_{i}}
                                                                 {u_{j}  +_{\L}  [t]  (\ou_{i})}   
                              \times   \displaystyle{\prod_{i=r - a + 1}^{r}}   
                                 \mathscr{S}^{\L} (E; 1/u_{i})   \right ).     
\end{array} 
\end{equation}
We would like to push-forward this  formula via the Gysin map 
\begin{equation*} 
    \pi^{r - a}_{*}:   MU^{*}(G^{1}(U_{n - r + a + 1})) 
 \longrightarrow 
MU^{*}(G^{1}(U_{n -r + a + 2})).  
\end{equation*} 
Taking (\ref{eqn:[x_1|b]^lambda_1(2)}) into account, we modify the right-hand side 
of (\ref{eqn:InductiveStepMU(G^1(U_n-r+a+1))}) 
as 
\begin{equation*} 
\begin{array}{lll} 
&  [u_{r - a + 1}^{n-1} \cdots u_{r-1}^{n-1} u_{r}^{n-1}]  
     \left  (     
         \displaystyle{\prod_{i=1}^{r-a-1}}    [x_{i}|\b]_{\L}^{\lambda_{i}}
              \prod_{j = i  +1}^{n}  (x_{i}  +_{\L}  [t] (\ox_{j})) 
         \right.       \medskip \\
  & \times     \dfrac{[x_{r-a}|\b]_{\L}^{\lambda_{r-a}}} {x_{r-a} +_{\L}  [t](\ox_{r-a})}  
                   \left  \{    \displaystyle{\prod_{j=r-a}^{n}} (x_{r-a}  +_{\L} [t](\ox_{j}))  
                                  -   \prod_{j=r-a}^{n}  [t] (x_{r-a} +_{\L}  \ox_{j})   
                   \right \}   \medskip \\
   &          \times    \displaystyle{\prod_{i=r-a + 1}^{r}} 
                          \dfrac{[u_{i}|\b]_{\L}^{\lambda_{i}}} {u_{i} +_{\L} [t] (\ou_{i})} 
                       \left \{ 
                                 \dfrac{1}{u_{i}  +_{\L}  [t] (\ox_{r-a})}    
                                       \prod_{j=r-a}^{n} (u_{i}  +_{\L}  [t] (\ox_{j}))  
                       \right.   
                                       \medskip \\
  &   \left.   \hspace{4.5cm}    -   \displaystyle{\prod_{j= r - a + 1}^{i - 1}} 
                                                      \dfrac{u_{i} +_{\L}  [t] (\ou_{j})} 
                                                               {[t] (u_{i} +_{\L}  \ou_{j})}  
                               \cdot  \dfrac{1}{[t](u_{i} +_{\L}  \ox_{r-a})}   
                    \prod_{j= r-a}^{n} [t]  (u_{i} +_{\L}  \ox_{j})  
            \right \}         \medskip 
\end{array}  
\end{equation*} 
\begin{equation*} 
\begin{array}{lll} 
   &   \times    \displaystyle{\prod_{i=r-a + 1}^{r}}  \prod_{j=1}^{r-a-1} 
           (u_{i}  +_{\L} \ox_{j})   \times \prod_{i=r-a+1}^{r} (u_{i} +_{\L} \ox_{r-a})    \times  
        \prod_{r - a + 1 \leq i < j \leq r}  \dfrac{ u_{j} +_{\L}  \ou_{i}}
                                                                 {u_{j}  +_{\L}  [t]  (\ou_{i})}      \medskip \\
      &   \left.    \hspace{8.5cm}  
         \times   \displaystyle{\prod_{i=r - a + 1}^{r}}       
                      \mathscr{S}^{\L}  (E; 1/u_{i}) 
      \right ).   
\end{array} 
\end{equation*} 
Then,  apply the fundamental Gysin formula (\ref{eqn:FundamentalFormula(TypeA)(ComplexCobordism)}). 
In the above modification,  we divide  both denominator  and numerator of 
 $\dfrac{1}{u_{i} +_{\L} [t](\overline{x}_{r-a})}$ by $u_{i}$,  and consider it  as a formal power 
series in $x_{r-a}$.    We also treat 
$\dfrac{1}{[t] (u_{i} +_{\L} \overline{x}_{r-a})}$ in the same manner. 
Under this remark, the result  is just replacing $x_{r-a}$ by the formal variable $u_{r-a}$, and 
multiplying by $\mathscr{S}^{\L} (U_{n - r + a + 1}; 1/u_{r-a})$. 
Then, 
we extract the coefficient of  $u_{r-a}^{n - r + a}$.  
   Since we know from (\ref{eqn:SegreSeries(ComplexCobordism)})
\begin{equation*} 
    \mathscr{S}^{\L} (U_{n - r + a + 1}; 1/u_{r-a})  
     =  u_{r-a}^{-(r - a - 1)} \prod_{j=1}^{r-a-1} (u_{r-a} +_{\L}  \ox_{j})  \times \mathscr{S}^{\L} (E; 1/u_{r-a}), 
\end{equation*} 
we see directly that the formula (\ref{eqn:InductiveStepMU(G^1(U_n-r+a+1))}) holds 
for $a + 1$. 
Therefore, when $a = r$, we have 
\begin{equation*} 
\begin{array}{lll} 
  &  (\pi^{r, r-1, \ldots, 1})_{*} 
  \left ([\x|\b]_{\L}^{\lambda}  \displaystyle{\prod_{i=1}^{r}}  
                        \prod_{j=i + 1}^{n} (x_{i} +_{\L}  [t] (\ox_{j}))  
                           \right )   \medskip  \\  
  & =  [u_{1}^{n-1} \dots u_{r}^{n-1}] 
  \left [    \displaystyle{\prod_{i=1}^{r}} 
                         \dfrac{ [u_{i}|\b]_{\L}^{\lambda_{i}}}{u_{i} +_{\L}  [t] (\ou_{i})}    
        \left \{     \prod_{j=1}^{n} (u_{i} +_{\L}  [t] (\ox_{j}))   
            -  \prod_{j=1}^{i-1}  
                    \dfrac{u_{i} +_{\L} [t] (\ou_{j})} 
                                       {[t] (u_{i} +_{\L}  \ou_{j})}   
                         \prod_{j=1}^{n}[t]  (u_{i}+_{\L}   \ox_{j})  
          \right \}     \right.  \medskip \\ 
   & \hspace{7.5cm}   \left.  \times \displaystyle{\prod_{1 \leq i < j \leq r}}  
                  \dfrac{u_{j} +_{\L}  \ou_{i}} {u_{j} +_{\L} [t] (\ou_{i})}  \times 
      \prod_{i=1}^{r}  \mathscr{S}^{\L}  (E; 1/u_{i})     \right ]. 
\end{array} 
\end{equation*} 
Then, using the Segre series (\ref{eqn:SegreSeries(ComplexCobordism)}),  
we obtain the required  formula.  
\end{proof}  
By a characterization (\ref{eqn:CharacterizationUFH-LP}),  the left-hand side of (\ref{eqn:GFUFH-LPI}) 
is $HP^{\L}_{\lambda}(\x_{n}; t |\b)$, and hence the right-hand side gives a  generating function 
for $HP^{\L}_{\lambda}(\x_{n}; t|\b)$.

  Let us simplify this generating function in the following way: 
First note that 
\begin{equation*} 
    \displaystyle{\prod_{1 \leq i < j \leq r}} 
                                      \dfrac{u_{j} +_{\L}  \ou_{i}}  {u_{j} +_{\L}  [t]  (\ou_{i})}  
        =  \prod_{1 \leq j < i \leq r}   \dfrac{u_{i} +_{\L}  \ou_{j}}  {u_{i} +_{\L}  [t]  (\ou_{j})}  
        =  \prod_{i=1}^{r}  \prod_{j=1}^{i-1}   \dfrac{u_{i} +_{\L}  \ou_{j}}  {u_{i} +_{\L}  [t]  (\ou_{j})}.   
\end{equation*} 
Therefore if we put 
\begin{equation*} 
\begin{array}{llll} 
  &       {\mathcal{HP}}^{\L, (n)} _{i, \lambda_{i}}(u_{1}, u_{2},  \ldots, u_{i}|\bm{b})
   :=   \dfrac{u_{i}} {u_{i} +_{\L} [t] (\ou_{i})} 
                  \cdot  \dfrac{1}{\mathscr{P}^{\L} (u_{i})}     \medskip \\
 & \hspace{1cm}  \times \left (  
             \displaystyle{\prod_{j=1}^{n}}   \dfrac{u_{i} +_{\L}  [t] (\ox_{j})} 
                                                                        {u_{i} +_{\L}  \ox_{j}} 
                   \prod_{j=1}^{i-1}   \dfrac{u_{i} +_{\L}  \ou_{j}}  {u_{i} +_{\L}  [t]  (\ou_{j})}
                      \prod_{j=1}^{\lambda_{i}} \dfrac{u_{i} +_{\L}  b_{j}} {u_{i}} 
        \right.     \medskip   \\
    & \left.     \hspace{6.5cm}     -          \;           
                                       \displaystyle{\prod_{j=1}^{n}}  
                                           \dfrac{[t] (u_{i} +_{\L} \ox_{j})} 
                                           {u_{i} +_{\L} \ox_{j}}    
                                       \prod_{j  = 1}^{i-1}
                                                                \dfrac{u_{i} +_{\L} \ou_{j}} 
                                                                     {[t] (u_{i} +_{\L} \ou_{j})}    
                                        \prod_{j=1}^{\lambda_{i}} \dfrac{u_{i} +_{\L}  b_{j}} {u_{i}} 
                                              \right ),      \medskip   \\
  &  {\mathcal{HP}}^{\L, (n)}_{\lambda} (\bm{u}_{r}|\bm{b}) =  
    {\mathcal{HP}}^{\L, (n)}_{\lambda}   (u_{1},  u_{2}, \ldots, u_{r}|\bm{b}) 
      :=   \displaystyle{\prod_{i=1}^{r}}  
    {\mathcal{HP}}^{\L, (n)}_{i,  \lambda_{i}}(u_{1}, u_{2}, \ldots, u_{i}|\bm{b}), \medskip 
\end{array} 
\end{equation*} 
then,  one has 
\begin{equation}    \label{eqn:GFUFH-LP} 
      HP^{\L}_{\lambda}(\x_{n}; t|\b)  = [\bm{u}^{-\lambda}]  
     \left (  
                 {\mathcal{HP}}^{\L, (n)}_{\lambda}(\bm{u}_{r}|\bm{b})
     \right  ). 
\end{equation} 
Moreover,  observe that 
\begin{itemize} 
  \item    $\dfrac{u_{i}} {u_{i} +_{\L} [t] (\ou_{i})} 
                  \cdot  \dfrac{1}{\mathscr{P}^{\L} (u_{i})}$is a formal power series in  $u_{i}$.

  \item   $\displaystyle{\prod_{j=1}^{n}}  
                                           \dfrac{[t] (u_{i} +_{\L} \ox_{j})} 
                                           {u_{i} +_{\L} \ox_{j}}    
                                       \prod_{j  = 1}^{i-1}
                                                                \dfrac{u_{i} +_{\L} \ou_{j}} 
                                                                     {[t] (u_{i} +_{\L} \ou_{j})}$  is regarded as 
       a formal power series in $u_{i}$  with constant term $t^{n - i + 1}$.

  \item   $\displaystyle{\prod_{j=1}^{\lambda_{i}}}  \dfrac{u_{i} +_{\L}  b_{j}} {u_{i}}$  is a formal Laurent series 
            in $u_{i}$ whose lowest degree term is  $u_{i}^{-\lambda_{i}}$ with  coefficient 
         $\displaystyle{\prod_{j=1}^{\lambda_{i}}} b_{j}$.  
\end{itemize} 
Taking the above observation into account, we put
\begin{equation*} 
\begin{array}{llll} 
  &    \widetilde{\mathcal{HP}}^{\L, (n)}_{i,  \lambda_{i}}(u_{1}, u_{2},  \ldots, u_{i}|\bm{b})
   :=   \dfrac{u_{i}} {u_{i} +_{\L} [t] (\ou_{i})} 
                  \cdot  \dfrac{1}{\mathscr{P}^{\L} (u_{i})}     \medskip \\
 &   \hspace{3.3cm}   \times \left (  
             \displaystyle{\prod_{j=1}^{n}}   \dfrac{u_{i} +_{\L}  [t] (\ox_{j})} 
                                                                        {u_{i} +_{\L}  \ox_{j}} 
                   \prod_{j=1}^{i-1}   \dfrac{u_{i} +_{\L}  \ou_{j}}  {u_{i} +_{\L}  [t]  (\ou_{j})}
                      \prod_{j=1}^{\lambda_{i}} \dfrac{u_{i} +_{\L}  b_{j}} {u_{i}}   
                               -                   
                                         t^{n  - i + 1}  
                                        \prod_{j=1}^{\lambda_{i}} \dfrac{b_{j}} {u_{i}} 
                                              \right ),      \medskip   \\
  &    \widetilde{\mathcal{HP}}^{\L, (n)}_{\lambda}  (\bm{u}_{r}|\bm{b}) =   
            \widetilde{\mathcal{HP}}^{\L, (n)}_{\lambda}   (u_{1},  u_{2}, \ldots, u_{r}|\bm{b}) 
      :=   \displaystyle{\prod_{i=1}^{r}}    
                \widetilde{\mathcal{HP}}^{\L, (n)}_{i,  \lambda_{i}}(u_{1}, u_{2}, \ldots, u_{i}|\bm{b}).  \medskip 
\end{array} 
\end{equation*} 
Then, we can reduce $\mathcal{HP}^{\L, (n)}_{\lambda}  (\bm{u}_{r}|\bm{b})$ 
to $\widetilde{\mathcal{HP}}^{\L, (n)}_{\lambda}   (\bm{u}_{r}|\bm{b})$, 
and  we  obtain from (\ref{eqn:GFUFH-LP})   the following result:  
\begin{theorem} [Generating function for $HP^{\L}_{\lambda} (\bm{x}_{n}; t|\b)$]    \label{thm:GFUFH-LP}  
For a sequence of positive integers  $\lambda = (\lambda_{1}, \ldots, \lambda_{r})$ with $r \leq n$, 
 the universal factorial Hall--Littlewood $P$-function 
$HP^{\L}_{\lambda}(\x_{n}; t|\b)$ is  the 
coefficient of $\bm{u}^{-\lambda} =  u_{1}^{-\lambda_{1}} u_{2}^{-\lambda_{2}} \cdots u_{r}^{-\lambda_{r}}$ 
in $\widetilde{\mathcal{HP}}^{\L, (n)}_{\lambda}(u_{1},  u_{2}, \ldots, u_{r}|\bm{b})$.  Thus 
\begin{equation*} 
         HP^{\L}_{\lambda}(\x_{n}; t|\b) 
    =  [\bm{u}^{-\lambda}]  \left ( 
                                     \widetilde{\mathcal{HP}}^{\L, (n)}_{\lambda}   (\bm{u}_{r}|\bm{b}) 
                                  \right ). 
\end{equation*}   
\end{theorem}  

If we specialize the universal formal group law 
 $F_{\L}(u, v) = u +_{\L} v$ to  $F_{a}(u, v) = u + v$, the above generating function 
becomes  a relatively simple form:  
\begin{equation*}   \label{enq:GFFH-LP} 
\begin{array}{lll} 
      &  \widetilde{\mathcal{HP}}^{(n)}_{i, \lambda_{i}}(u_{1}, u_{2},  \ldots, u_{i}|\bm{b})
   =   \dfrac{1}{1-t}  
        \left (  
             \displaystyle{\prod_{j=1}^{n}}   \dfrac{u_{i}  - tx_{j}} 
                                                                        {u_{i}  -  x_{j}} 
                   \prod_{j=1}^{i-1}   \dfrac{u_{i}  -  u_{j}}  {u_{i}  - tu_{j}}
                      \prod_{j=1}^{\lambda_{i}} \dfrac{u_{i} +  b_{j}} {u_{i}}   
                               -                   
                                         t^{n  - i + 1}  
                                        \prod_{j=1}^{\lambda_{i}} \dfrac{b_{j}} {u_{i}} 
                                              \right ),      \medskip   \\
  &   \widetilde{\mathcal{HP}}^{(n)}_{\lambda}  (\bm{u}_{r}|\bm{b}) =   
        \widetilde{\mathcal{HP}}^{(n)}_{\lambda}   (u_{1},  u_{2}, \ldots, u_{r}|\bm{b}) 
      =  \displaystyle{\prod_{i=1}^{r}}   
            \widetilde{\mathcal{HP}}^{(n)}_{i,  \lambda_{i}}(u_{1}, u_{2}, \ldots, u_{i}|\bm{b}).  \medskip 
\end{array} 
\end{equation*} 
Thus we have  the following corollary: 
\begin{cor}[Generating function for $HP_{\lambda}(\bm{x}_{n}; t|\b)$]     \label{cor:GFFH-LP}  
For a sequence of positive integers  $\lambda = (\lambda_{1}, \ldots, \lambda_{r})$ with $r \leq n$, 
the factorial Hall--Littlewood $P$-polynomial 
$HP_{\lambda}(\x_{n}; t|\b)$ is  the 
coefficient of $\bm{u}^{\lambda} =   u_{1}^{-\lambda_{1}} u_{2}^{-\lambda_{2}} \cdots u_{r}^{-\lambda_{r}}$ 
in  $\widetilde{\mathcal{HP}}^{(n)}_{\lambda}(u_{1},  u_{2}, \ldots, u_{r}|\bm{b})$. Thus 
\begin{equation*} 
            HP_{\lambda}(\bm{x}_{n}; t|\b)  = [\bm{u}^{-\lambda}] \left ( 
                                                   \widetilde{\mathcal{HP}}^{(n)}_{\lambda}  (\bm{u}_{r}|\bm{b}) 
                                                                                    \right ).  
\end{equation*}
\end{cor} 

 \subsection{Generating function for $HQ^{\L}_{\lambda}(\x_{n} ; t |\b)$}
Next we shall derive the generating function for $HQ^{\L}_{\lambda}(\x_{n} ; t |\b)$.   
 In the  one-row case $\lambda = (\lambda_{1})$  of (\ref{eqn:CharacterizationUFH-LQ}), 
we push-forward the expression 
\begin{equation*} 
\begin{array}{lll} 
     [[x_{1};  t|\b]]_{\L}^{\lambda_{1}} \displaystyle{\prod_{j=2}^{n}}   (x_{1}  +_{\L}  [t] (\ox_{j}))  
 &  =  \dfrac{ (x_{1} +_{\L}  [t] (\ox_{1})) [x_{1}|\b]_{\L}^{\lambda_{1} - 1} } {x_{1} +_{\L}  [t] (\ox_{1})} 
        \displaystyle{\prod_{j=1}^{n}}   (x_{1} +_{\L}  [t] (\ox_{j}))   \medskip \\
  & = [x_{1}|\b]_{\L}^{\lambda_{1} - 1}  \displaystyle{\prod_{j=1}^{n}}   (x_{1} +_{\L}  [t] (\ox_{j})),   \medskip 
\end{array}  
\end{equation*}  
which is a formal power series in $x_{1}$, 
and therefore (\ref{eqn:FundamentalFormula(TypeA)(ComplexCobordism)})  applies without any problem,  
 and one   can  obtain the following: 
\begin{equation*} 
\begin{array}{llll} 
 &   HQ^{\L}_{\lambda}(\x_{n}; t|\b)   \medskip \\ 
         & =  [u_{1}^{-\lambda_{1}} \cdots u_{r}^{-\lambda_{r}}]   
      \left ( 
              \displaystyle{\prod_{i=1}^{r}}  \dfrac{1}{\mathscr{P}^{\L} (u_{i})}     
               \prod_{j=1}^{n}  \dfrac{u_{i} +_{\L}   [t](\overline{x}_{j})}  {u_{i} +_{\L} \overline{x}_{j}} 
              \prod_{1 \leq i < j \leq r}   \dfrac{u_{j} +_{\L} \overline{u}_{i}} 
                                                                   {u_{j} +_{\L} [t](\overline{u}_{i})}  
         \times \prod_{i=1}^{r} \prod_{j=1}^{\lambda_{i}-1}  \dfrac{u_{i} +_{\L} b_{j}}{u_{i}} 
    \right ).   \medskip \\
\end{array} 
\end{equation*}  
For  each non-negative  integer $k$, we set 
\begin{equation*} 
 \mathcal{HQ}^{\L, (n)}_{k}(u|\b)   
            :=     
                                    \dfrac{1}{\mathscr{P}^{\L} (u)} 
                                     \displaystyle{\prod_{j=1}^{n}}   \dfrac{u +_{\L} [t](\overline{x}_{j})} 
                                                                {u +_{\L} \overline{x}_{j}}    
       \times \prod_{j=1}^{k} \dfrac{u +_{\L}  b_{j}}{u}.  
\end{equation*} 
For a sequence of positive integers 
 $\lambda = (\lambda_{1}, \ldots, \lambda_{r})$ with $r \leq n$, 
we set 
\begin{equation*}   
      \mathcal{HQ}^{\L, (n)}_{\lambda}  (\bm{u}_{r}|\b)
    = \mathcal{HQ}^{\L, (n)}_{\lambda}   (u_{1},  u_{2},  \ldots, u_{r}|\b)  
             :=  \displaystyle{\prod_{i=1}^{r}}  \mathcal{HQ}^{\L, (n)}_{\lambda_{i}-1} (u_{i}|\b)     
         \displaystyle{\prod_{1 \leq i < j \leq r}}   \dfrac{u_{j} +_{\L} \overline{u}_{i}} 
                                                                   {u_{j} +_{\L} [t](\overline{u}_{i})}.    
\end{equation*} 
Thus we have the following result: 
\begin{theorem} [Generating function for $HQ^{\L}_{\lambda} (\bm{x}_{n}; t | \bm{b})$] 
  \label{thm:GFUFH-LQ}   
  For a sequence of positive integers  $\lambda = (\lambda_{1}, \ldots, \lambda_{r})$
with $r \leq n$,   
 the universal factorial Hall--Littlewood $Q$-function 
$HQ^{\L}_{\lambda}(\x_{n}; t|\b)$ is the coefficient of 
$\bm{u}^{-\lambda} = u_{1}^{-\lambda_{1}} u_{2}^{-\lambda_{2}} \cdots u_{r}^{-\lambda_{r}}$ 
in     $\mathcal{HQ}^{\L, (n)}_{\lambda}  (u_{1}, u_{2}, \ldots, u_{r}|\b)$. Thus 
\begin{equation*} 
    HQ^{\L}_{\lambda}(\x_{n}; t|\b)  = [\bm{u}^{-\lambda}]  
    \left ( \mathcal{HQ}^{\L, (n)}_{\lambda}   (\bm{u}_{r}|\b )  \right ). 
\end{equation*}  
\end{theorem} 

\section{Application  of generating functions}   \label{sec:ApplicationGF} 
\subsection{$e$-Cancellation property} 
A symmetric polynomial $f(x_{1}, \ldots, x_{n})$ with coefficients in 
$\Z$   has the {\it $Q$-cancellation property}  if the following holds: 
when the substitution $x_{1} = a$, $x_{2} = -a$, $a$ an indeterminate, 
is made in $f$,  the resulting polynomial is independent of $a$  
(Pragacz \cite[\S 2]{Pragacz1991}). 
It is known that the Schur $P$- and $Q$-polynomials 
  satisfy  this cancellation property. 
The notion of the $Q$-cancellation property is generalized in the following way: 
Let $e \geq 2$ be a fixed positive integer, and 
$\zeta = \zeta_{e}$ be  a primitive $e$th root of unity.  
We define a sequence 
$\bm{a}^{e}(\zeta) :=  (a, \zeta a,  \zeta^{2} a, \ldots,  \zeta^{e-1} a)$. 
Suppose that $e \leq n$. Then a symmetric polynomial $f(x_{1}, \ldots, x_{n})$ with coefficients 
in $\Z [\zeta]$ has the {\it $e$-cancellation property} 
if $f(\bm{a}^{e}(\zeta), x_{e+1}, \ldots, x_{n}) 
= f(a,  \zeta a,  \zeta^{2} a, \ldots,  \zeta^{e-1} a, x_{e+1}, \ldots, x_{n})$ 
does not depend on $a$.   In the case $e = 2$,  this property is nothing but 
the $Q$-cancellation property.  
By specializing $t$ to be $\zeta$, the factorial Hall--Littlewood 
polynomials $HP_{\lambda}(\x_{n}; \zeta|\b)$ and 
$HQ_{\lambda}(\x_{n}; \zeta|\b)$ are symmetric polynomials 
 with coefficients 
in $\Z [\zeta] \otimes \Z[\b] = \Z [\zeta] \otimes \Z[b_{1}, b_{2}, \ldots]$. 
Thus one can ask if these symmetric polynomials have 
the $e$-cancellation property or not.  
In this subsection,   as the  first application of our generating functions,  
we  shall  show 
the  $e$-cancellation property  of the factorial Hall--Littlewood $P$- and 
$Q$-polynomials. 
\begin{prop} [$e$-Cancellation property]    \label{prop:e-CancellationProperty}  
  Assume that $e \leq n$.  
  The factorial Hall--Littlewood polyonomials $HP_{\lambda}(\x_{n}; \zeta|\b)$ 
and $HQ_{\lambda}(\x_{n}; \zeta |\b)$ have the $e$-cancellation property.   
\end{prop} 
\begin{proof} 
Let $r$ be the length of $\lambda$. 
Then, by Corollary \ref{cor:GFFH-LP}, $HP_{\lambda}(\x_{n}; \zeta |\b)$ 
 is given as the  coefficient of  the following generating function
\begin{equation}   \label{eqn:GFFH-LP}  
   \dfrac{1}{(1-\zeta)^{r}}  
    \prod_{i=1}^{r}  
 \left (  
             \displaystyle{\prod_{j=1}^{n}}   \dfrac{u_{i}  - \zeta x_{j}} 
                                                                        {u_{i}  -  x_{j}} 
                   \prod_{j=1}^{i-1}   \dfrac{u_{i}  -  u_{j}}  {u_{i}  - \zeta u_{j}}
                      \prod_{j=1}^{\lambda_{i}} \dfrac{u_{i} +  b_{j}} {u_{i}}   
                               -                   
                                         \zeta^{n  - i + 1}  
                                        \prod_{j=1}^{\lambda_{i}} \dfrac{b_{j}} {u_{i}} 
  \right ). 
\end{equation} 
Substituting $(x_{1},  \ldots, x_{e})$ with $\bm{a}^{e}(\zeta)$ in each 
factor  $\displaystyle{\prod_{j=1}^{n}}  \dfrac{u_{i} - \zeta x_{j}} {u_{i} - x_{j}}$,  
we have 
\begin{equation*} 
   \prod_{j=1}^{e} \dfrac{u_{i} -  \zeta^{j} a} {u_{i} -  \zeta^{j-1} a}
  \times \prod_{j=e  + 1}^{n}   \dfrac{u_{i} - \zeta x_{j}} {u_{i} - x_{j}}  
=  \prod_{j=e  + 1}^{n}   \dfrac{u_{i} - \zeta x_{j}} {u_{i} - x_{j}}, 
\end{equation*} 
since $\zeta^{e} = 1$. Therefore, (\ref{eqn:GFFH-LP}) does not depend on 
$a$ and $x_{1},  \ldots, x_{e}$.  From this, the $e$-cancellation property 
of $HP_{\lambda}(\x_{n}; \zeta|\b)$ follows.  
By Theorem \ref{thm:GFUFH-LQ},  $HQ_{\lambda}(\x_{n}; t |\b)$ 
 is given as the  coefficient of  the following generating function
\begin{equation*}    
  \prod_{i=1}^{r}  \displaystyle{\prod_{j=1}^{n}}   \dfrac{u  - tx_{j}} 
                                                                {u  -  x_{j}}    
       \times \prod_{j=1}^{\lambda_{i}-1} \dfrac{u +   b_{j}}{u}
    \times  \displaystyle{\prod_{1 \leq i < j \leq r}}   \dfrac{u_{j}  - u_{i}} 
                                                                   {u_{j}  - tu_{i}}.     
\end{equation*} 
From this,  the $e$-cancellation property of $HQ_{\lambda}(\x_{n}; \zeta|\b)$ 
follows by the same reason as that of $HP_{\lambda}(\x_{n}; \zeta |\b)$.  
\end{proof}

\begin{rem} 
In the case of the universal factorial Hall--Littlewood $P$- and $Q$-functions, 
we consider the substitution $\bm{x}_{n}$ with the sequence $\bm{a}^{e}[\zeta] 
:= (a, [\zeta](a),  [\zeta^2](a),   \ldots, [\zeta^{e-1}](a))$.  Then, 
using Theorems $\ref{thm:GFUFH-LP}$ and  $\ref{thm:GFUFH-LQ}$, one can verify  easily 
that $HP^{\L}_{\lambda}(\bm{x}_{n}; \zeta |\bm{b})$ and $HQ^{\L}_{\lambda}(\bm{x}_{n}; \zeta |\bm{b})$ 
satisfy the $e$-cancellation property too.  
\end{rem} 

\subsection{Pfaffian formula for $GQ_{\nu} (\bm{x}_{n}|\b)$}      \label{subsec:PfaffianFormulaFGQP}
Another application  of  our generating functions, 
we shall  derive  the 
{\it Pfaffian formulas} for the  the   $K$-theoretic factorial $Q$-polynomial   $GQ_{\nu}(\x_{n}|\b)$,   
which  seems to be new.   
In what follows,  we assume that the length $\ell (\nu)$ of a  strict partition 
$\nu$ is $2m$ (even).  
We consider the specialization from  $F_{\L}(u, v) = u +_{\L} v  $ to $F_{m} (u, v) =  u  \oplus v$
with $t = -1$. 
Then $HQ^{\L}_{\nu}(\bm{x}_{n}; t |\bm{b})$ specializes to  $GQ_{\nu}(\bm{x}_{n}|\bm{b})$, 
and,  the generating function  $\mathcal{HQ}^{\L, (n)}_{\nu} (\bm{u}_{2m}|\b)$
 reduces to 
\begin{equation}  \label{eqn:GFFGQ}   
    \mathcal{GQ}^{(n)}_{\nu}  (\bm{u}_{2m} |\b)
      = \prod_{i=1}^{2m} \mathcal{GQ}^{(n)}_{\nu_{i}-1}  (u_{i}|\b)  
                 \prod_{1 \leq i  < j \leq 2m}   \dfrac{u_{j} \ominus u_{i}} 
                                                                {u_{j} \oplus u_{i}}, 
\end{equation} 
where, for each non-negative integer $k$,  we define
\begin{equation*} 
\begin{array}{lll}  
   \mathcal{GQ}^{(n)}_{k}(u|\b)   & :=     
                  \dfrac{1}{ 1 + \beta u}  \displaystyle{\prod_{j=1}^{n}}   \dfrac{u \oplus x_{j}} 
                                                                                        {u \ominus x_{j}} 
                     \times \prod_{j=1}^{k}  \dfrac{u \oplus b_{j}}{u}   \medskip \\
    & =   \dfrac{1} {1 + \beta u} 
                     \displaystyle{\prod_{j=1}^{n}} 
                                \dfrac{1 + (u^{-1} + \beta) x_{j}} {1  + (u^{-1} + \beta) \overline{x}_{j}} 
                  \times \prod_{j=1}^{k}  \{ 1 + (u^{-1} + \beta) b_{j} \}.  
\end{array}   
\end{equation*} 
This is a generating function for the factorial $K$-theoretic $Q$-polynomials $GQ_{\nu}(\bm{x}_{n}|\bm{b})$.  
Here we recall from Ikeda--Naruse \cite[Lemma 2.4]{Ikeda-Naruse2013} the following formula: 
\begin{equation*} 
      \mathrm{Pf}  \, \left ( \dfrac{x_{j} - x_{i}} 
                                             {x_{j} \oplus x_{i}} 
                            \right )_{1 \leq i < j \leq 2m} 
     =      \prod_{1 \leq  i  < j \leq 2m}  \dfrac{x_{j} - x_{i}} {x_{j} \oplus x_{i}}. 
\end{equation*} 
Thus we can compute 
\begin{equation*} 
\begin{array}{llll}  
  &    \mathcal{GQ}^{(n)}_{\nu}(\bm{u}_{2m}|\b)
 =  \displaystyle{\prod_{i=1}^{2m}}   \mathcal{GQ}^{(n)}_{\nu_{i}-1}   (u_{i}|\b)
    \displaystyle{\prod_{1 \leq i < j \leq 2m}}   
                  \dfrac{u_{j} \ominus u_{i}} 
                           {u_{j} \oplus u_{i}}     \medskip  \\ 
    & =  \displaystyle{\prod_{i=1}^{2m}}  \mathcal{GQ}^{(n)}_{\nu_{i}-1}  (u_{i}|\b)  
   \displaystyle{\prod_{i=1}^{2m}}   \dfrac{1}{(1 + \beta u_{i})^{2m -i}}   
                             \cdot  \mathrm{Pf} \,  
                                 \left (  
                                  \dfrac{u_{j} -  u_{i}} 
                                           {u_{j}  \oplus  u_{i}} 
                                  \right )_{1 \leq i < j \leq 2m}       \medskip \\
   &  =   \mathrm{Pf} \, \left (  
                                             \mathcal{GQ}^{(n)}_{\nu_{i}-1}  (u_{i}|\b)  \; 
                                            \mathcal{GQ}^{(n)}_{\nu_{j}-1} (u_{j}|\b)
                                              \dfrac{1}{(1 + \beta  u_{i})^{2m - i}}  
                                              \dfrac{1}{(1 + \beta  u_{j})^{2m - j}}     
                                           \cdot   \dfrac{u_{j} - u_{i}} 
                                           {u_{j}  \oplus u_{i}} 
                           \right )_{1 \leq i < j \leq 2m}   \medskip \\
&  =   \mathrm{Pf}  \left (  
                                      (1 + \beta u_{i})^{i + 1 - 2m} (1 + \beta u_{j})^{j - 2m} 
                                        \mathcal{GQ}^{(n)}_{\nu_{i}-1}  (u_{i}|\b) \; 
                                        \mathcal{GQ}^{(n)}_{\nu_{j}-1} (u_{j}|\b)   \cdot 
                                       \dfrac{u_{j}  \ominus u_{i}} {u_{j}  \oplus  u_{i}}
                             \right )_{1 \leq i < j \leq 2m}.       \medskip 
\end{array} 
\end{equation*} 
For non-negative integers $p, q \geq 0$ and positive integers $k, l \geq 1$, 
we define  polynomials $GQ_{(k, l)}^{(p, q)}(\x_{n}|\b)$ to be 
\begin{equation*} 
   GQ_{(k, l)}^{(p, q)} (\x_{n}|\b) :=  [u_{1}^{-k} u_{2}^{-l}] 
          \left (    \mathcal{GQ}^{(n)}_{p-1}(u_{1}|\b)  \;  
                     \mathcal{GQ}^{(n)}_{q-1} (u_{2}|\b)  
           \cdot \dfrac{u_{2} \ominus u_{1}} {u_{2} \oplus u_{1}}  
         \right ).  
\end{equation*} 
Note that, by Theorem \ref{thm:GFUFH-LQ},  we have 
$GQ_{(k, l)} (\x_{n}|\b)    =  GQ_{(k, l)}^{(k, l)} (\x_{n}|\b)$ for 
positive integers $k > l > 0$.  
Then, by Theorem  \ref{thm:GFUFH-LQ},  one obtains 
\begin{equation*} 
\begin{array}{llll} 
  & GQ_{\nu} (\x_{n}|\b)   =  \left [ \displaystyle{\prod_{i=1}^{2m}} u_{i}^{-\nu_{i}} \right ] 
                                  \Biggl (   \mathrm{Pf} \, \Biggl (
                                           (1 + \beta u_{i})^{i + 1  - 2m} 
                                           (1 + \beta u_{j})^{j - 2m}   
                                            \medskip \\
          &    \hspace{6cm}  \times    \;  \mathcal{GQ}^{(n)}_{\nu_{i}-1} (u_{i}|\b) \;  
                                                     \mathcal{GQ}^{(n)}_{\nu_{j}-1} (u_{j}|\b)
                          \times  \left. \left. 
                                   \dfrac{u_{j} \ominus u_{i}} 
                                           {u_{j}  \oplus u_{i}} 
                                         \right )_{1 \leq  < j \leq 2m} \right )   
                               \medskip \\
  & =  \mathrm{Pf}  
                                 \left (   [u_{i}^{-\nu_{i}} u_{j}^{-\nu_{j}} ]  
                                            \left ( 
                                              (1 + \beta  u_{i})^{i + 1 - 2m}  
                                              (1 + \beta  u_{j})^{j - 2m}   
                                              \mathcal{GQ}^{(n)}_{\nu_{i}-1} (u_{i}|\b) \;  
                                              \mathcal{GQ}^{(n)}_{\nu_{j}-1} (u_{j}|\b) 
                  \right. \right.    \medskip \\
     &    \left. \left.   \hspace{11.6cm}  
                                  \times         \dfrac{u_{j} \ominus  u_{i}} 
                                                        {u_{j}  \oplus  u_{i}} 
                                           \right )  
                                   \right )_{1 \leq i < j \leq 2m}    \medskip \\
      & =  \mathrm{Pf} 
                \left ( 
                      [u_{i}^{-\nu_{i}} u_{j}^{-\nu_{j}} ] 
                          \left ( 
                                 \displaystyle{\sum_{k=0}^{\infty}  \sum_{l=0}^{\infty}} 
                                       \binom{i + 1 - 2m}{k}  \binom{j - 2m}{l}  
                     \beta^{k +l}  u_{i}^{k} u_{j}^{l}  \right.   \right. \medskip \\
    &   \hspace{6.3cm}  
                             \left.    \left.  
                                        \times  \;     \mathcal{GQ}^{(n)}_{\nu_{i}-1} (u_{i}|\b) \;  
                                                          \mathcal{GQ}^{(n)}_{\nu_{j}-1} (u_{j}|\b)
                                            \cdot  \dfrac{u_{j} \ominus u_{i}} 
                                           {u_{j}  \oplus u_{i}}  
                          \right ) 
                \right )_{1 \leq i < j \leq 2m}    \medskip \\
   & =   \mathrm{Pf}  \left ( 
                                            \displaystyle{\sum_{k=0}^{\infty} \sum_{l=0}^{\infty}}   
                                              \displaystyle{\binom{i + 1 - 2m}{k}  \binom{j - 2m}{l}} 
                                                  \beta^{k + l} 
                                                          GQ_{(\nu_{i} + k,  \nu_{j} + l)}^{(\nu_{i}, \nu_{j})} (\x_{n}|\b)  
                              \right )_{1 \leq i < j \leq 2m}. 
\end{array} 
\end{equation*} 
Thus we obtained the following: 
\begin{theorem}  [Pfaffian formula for $GQ_{\nu}(\x_{n}|\b)$]  \label{thm:PfaffianFormulaGQ}
For a strict partition $\nu$ of length $2m$,  we have 
\begin{equation*}  
   GQ_{\nu}  (\x_{n} |\b)  =  \mathrm{Pf}  \left ( 
                                            \displaystyle{\sum_{k=0}^{\infty} \sum_{l=0}^{\infty}}   
                                              \displaystyle{\binom{i + 1 - 2m}{k}  \binom{j - 2m}{l}} 
                                                  \beta^{k + l} 
                                                          GQ_{(\nu_{i} + k,  \nu_{j} + l)}^{(\nu_{i}, \nu_{j})} (\x_{n}|\b)  
                              \right )_{1 \leq i < j \leq 2m}. 
\end{equation*}   
\end{theorem}   

\begin{rem}    \label{rem:GFGQ}   
\begin{enumerate} 
\item  Putting $\bm{b} = \bm{0}$ in $(\ref{eqn:GFFGQ})$,    
we  obtain  a generating function for the $($non-factorial$)$ 
$K$-theoretic $Q$-polynomials $GQ_{\nu}  (\x_{n})$.  
On the other hand,  dual $K$-theoretic $P$- and $Q$-polynomials was  introduced 
in our previous papers \cite[\S 5]{Nakagawa-Naruse2016},  \cite{Nakagawa-Naruse2018(arXiv)}.  
We have a conjecture on a generating function for 
the dual $K$-theoretic $Q$-polynomials, and their Pfaffian formula 
$($see Appendix  \S $\ref{subsec:ConjectureGFgq})$.

\item  The generating function technique can also be applied to the derivation of 
the determinantal formula for factorial Grothendieck 
polynomials $G_{\lambda}(\bm{x}_{n}|\bm{b})$.
On the other hand, a generating function for the dual Grothendieck polynomials $g_{\lambda}(\bm{x}_{n})$
$($for their definition, see Lascoux--Naruse \cite{Lascoux-Naruse2014}$)$ can be obtained by a 
purely algebraic manner.  We shall give the details in the Appendix, \S $\ref{subsec:GFDGP}$.   
\end{enumerate}  
\end{rem}

\section{Appendix}   \label{sec:Appendix}   
\subsection{Generating function for the dual Grothendieck polynomials}    \label{subsec:GFDGP}  
As we mentioned in Remark \ref{rem:GFGQ},  we give a generating function 
for the dual Grothendieck polynomials.   
Following  Lascoux--Naruse \cite{Lascoux-Naruse2014},   let us  introduce the 
dual  Grothendieck polynomials   $g_{\lambda}(\y_{n})$, where 
$\y_{n} = (y_{1}, y_{2}, \ldots, y_{n})$ is a  set of independent variables 
 and $\lambda \in \mathcal{P}_{n}$.  
First we need some notation: 
Given two sets of variables  (called {\it alphabets} as usual) 
   $\mathbf{A}$, $\mathbf{B}$, the {\it complete functions}  $s_{k} (\mathbf{A} - \mathbf{B}) 
  \; (k  = 0, 1, 2, \ldots)$ are given by the following generating function: 
          \begin{equation*} 
            \sum_{k=0}^{\infty} s_{k}(\mathbf{A} - \mathbf{B}) z^{k}    =  \prod_{a \in \mathbf{A}} \dfrac{1}{1 - az}  
                                                              \prod_{b \in \mathbf{B}}  (1 - bz).    
           \end{equation*} 
     In particular,  when we add  $r$ letters specialized to $1$, namely the set  
           $\{  \underbrace{1, 1, \ldots, 1}_{r} \}$,  to one of the alphabets $\mathbf{A}$ or $\mathbf{B}$, 
we have 
     \begin{equation*}    \label{enq:s(A-Bpmr)}
           \sum_{k=0}^{\infty}  s_{k} (\mathbf{A} - \mathbf{B} \pm r)  z^{k}    
        =   (1 - z)^{\mp r}  \prod_{a \in \mathbf{A}} \dfrac{1}{1 - az}   \prod_{b \in \mathbf{B}} (1 - bz).  
   \end{equation*} 
Then,  for the variables $\y_{n} = (y_{1}, \ldots, y_{n})$ and any integer $r$, we have 
  \begin{equation*}  
  \begin{array}{llll}  
    \displaystyle{\sum_{k=0}^{\infty}}   s_{k} (\y_{n} + r)  z^{k} 
   & =  (1 - z)^{-r}   \displaystyle{\prod_{i=1}^{n}}  \dfrac{1}{1 - y_{i}z}    \medskip \\  
   & =   \left (   \displaystyle{\sum_{i =0}^{\infty}}  \binom{-r}{i}  (-z)^{i}  \right ) 
            \left ( \displaystyle{\sum_{j=0}^{\infty}}   h_{j}(\y_{n})  z^{j}  \right )  \medskip  \\ 
  & =    \left (  \displaystyle{\sum_{i=0}^{\infty}}   (-1)^{i} \binom{r + i - 1}{i} (-z)^{i}  \right ) 
           \left (  \displaystyle{\sum_{j=0}^{\infty}}   h_{j}(\y_{n})  z^{j}  \right )   \medskip   \\ 
   & =    \displaystyle{\sum_{k=0}^{\infty}}  
              \left (  \sum_{i = 0}^{k}  \binom{r + i-1}{i}  h_{k-i}  (\y_{n})  \right )  z^{k},   \medskip 
  \end{array}  
\end{equation*} 
and hence we have 
\begin{equation}   \label{eqn:s_k(y_n+r)}  
     s_{k}  (\y_{n}  + r)   =  \sum_{i=0}^{k}  \binom{r + i-1}{i}  h_{k-i} (\y_{n}) \quad 
(k = 0, 1, \ldots).  
\end{equation} 
Using  (\ref{eqn:s_k(y_n+r)}), the dual  Grothendieck polynomial 
$g_{\lambda} (\y_{n})$ for $\lambda \in \mathcal{P}_{n}$  of length $r$  is given 
by (see Lascoux--Naruse \cite[(3)]{Lascoux-Naruse2014}):  
          \begin{equation}    \label{eqn:DefinitionDSGP}
          \begin{array}{llll}  
         g_{\lambda}(\y_{n})   &  =  s_{\lambda}(\y_{n}, \y_{n} + 1,  \ldots, \y_{n} + n-1)   \medskip \\
                                      & =  \det  (s_{\lambda_{i} - i + j} (\y_{n} + i-1))_{1 \leq i, j \leq r}  \medskip \\
                                 & =  \det \left (  
                                                                \displaystyle{\sum_{k=0}^{\lambda_{i} -i + j}}
                                                                    \binom{(i-1) + k-1}{k}  
                                                                               h_{\lambda_{i} - i + j - k} (\y_{n})  
                                                 \right )_{1 \leq i, j \leq r} \medskip \\
                                  & = \det \left ( 
                                                     \displaystyle{\sum_{k =  0}^{\infty}}  
                                                              \binom{i + k-2}{k}  h_{\lambda_{i} - i + j - k} (\y_{n}) 
                                               \right )_{1 \leq i, j \leq r}.    
           \end{array}
            \end{equation}

We set 
\begin{equation*} 
\begin{array}{llll} 
    H(z) & = H^{(n)}(z) 
       :=   \displaystyle{\prod_{j=1}^{n}}    \dfrac{1}{1 - y_{j}z}     
         =  \displaystyle{\sum_{k=0}^{\infty}}  h_{k}(\y_{n}) z^{k},  \medskip \\
    g (\bm{z}_{r})  & = g(z_{1}, \ldots, z_{r}) 
         :=       \displaystyle{\prod_{i=1}^{r}}  H(z_{i})  \prod_{1 \leq i < j \leq r} \dfrac{z_{i} \ominus z_{j}}{z_{i}}.  
\end{array} 
\end{equation*} 
We shall show that $g(\bm{z}_{r})$ is the generating function for the dual  Grothendieck polynomials, namely we have the following: 
\begin{theorem}[Generating function for $g_{\lambda}(\y_{n})$]     \label{thm:GFDGP} 
  For a partition $\lambda = (\lambda_{1}, \ldots, \lambda_{r})$ of length $\ell (\lambda) 
= r \leq n$, the dual Grothendieck polynomial $g_{\lambda}(\y_{n})$ is the 
coefficient of $\bm{z}^{\lambda} = z_{1}^{\lambda_{1}} z_{2}^{\lambda_{2}} \cdots z_{r}^{\lambda_{r}}$ in $g (z_{1}, z_{2}, \ldots, z_{r})$. Thus 
\begin{equation*}
            g_{\lambda} (\y_{n})  = [\bm{z}^{\lambda}]  (g (\bm{z}_{r})). 
\end{equation*} 
\end{theorem} 
\begin{proof}  
By the Vandermonde determinant formula, we have 
\begin{equation*} 
\begin{array}{lll} 
    \displaystyle{\prod_{1 \leq i < j \leq r}}  \dfrac{z_{i} \ominus z_{j}} {z_{i}} 
   & =   \displaystyle{\prod_{i=1}^{r-1}} \prod_{j=i + 1}^{r}  \dfrac{1}{1  + \beta z_{j}} \cdot 
                                                            \dfrac{z_{i} - z_{j}}{z_{i}}  \medskip \\  
   &  =  \displaystyle{\prod_{i=1}^{r-1}}  \dfrac{1}{(1 + \beta z_{i+1}) \cdots (1 + \beta z_{r})}  
            \cdot \prod_{i=1}^{r-1}  \dfrac{1}{z_{i}^{r-i}}  
           \cdot \prod_{1 \leq i  < j \leq r}  (z_{i} - z_{j})  \medskip \\ 
   &  =  \displaystyle{\prod_{i=1}^{r}} \dfrac{1}{(1 + \beta z_{i})^{i-1}}  
                \cdot \prod_{i=1}^{r} \dfrac{1}{z_{i}^{r-i}}  \cdot \det (z_{i}^{r-j})_{1 \leq i, j \leq r}   \medskip \\ 
   & =   \det  ((1 + \beta z_{i})^{1-i}  z_{i}^{i-j})_{1 \leq i, j \leq r}.  \medskip  
\end{array} 
\end{equation*} 
Therefore one can compute 
\begin{equation*}  
   g(z_{1}, \ldots, z_{r})  =   \displaystyle{\prod_{i=1}^{r}} H(z_{i})  \prod_{1 \leq i < j \leq r}  \dfrac{z_{i} \ominus z_{j}} {z_{i}}  
   =  \det  ((1 + \beta z_{i})^{1-i}  z_{i}^{i-j}  H(z_{i}))_{1 \leq i, j \leq r}.   
\end{equation*} 
Extracting the  coefficient of the monomial  $\bm{z}^{\lambda}  = \prod_{i=1}^{r} z_{i}^{\lambda_{i}}$, 
  we obtain 
\begin{equation*} 
\begin{array}{llll}  
   [\bm{z}^{\lambda}]  (g (z_{1}, \ldots, z_{r}))   
& =    \left  [ \displaystyle{\prod_{i=1}^{r}}  z_{i}^{\lambda_{i}}   \right ]  
        \left  (  \det  ((1 + \beta z_{i})^{1 - i} z_{i}^{i-j}  H(z_{i}))_{1 \leq i, j \leq r}  \right ) \medskip \\
 & =   \det   \Big (   
                                    [z_{i}^{\lambda_{i}} ] ((1 + \beta z_{i})^{1-i}  z_{i}^{i-j}  H(z_{i})  
                 \Big )_{1 \leq i, j \leq r}   \medskip \\
  & =  \det   \left  ( [z_{i}^{\lambda_{i}}]  \left   (
                                 \displaystyle{\sum_{k =  0}^{\infty} } \binom{1-i}{k} \beta^{k} z_{i}^{k}  
                  \cdot z_{i}^{i-j}  H(z_{i})  
                                                      \right )      
                    \right )_{1 \leq i, j \leq r}    \medskip    \\
   & =   \det \left (   
                            \displaystyle{\sum_{k =  0}^{\infty}}   \binom{1-i}{k}  \beta^{k}  
                                                h_{\lambda_{i} - i + j - k} (\y_{n}) 
                     \right )_{1 \leq i, j \leq r}   \medskip  \\
  & = \det  \left (  
                      \displaystyle{\sum_{k = 0}^{\infty}}   \binom{i + k-2}{k}  (-\beta)^{k}  h_{\lambda_{i} - i + j - k} (\y_{n})  
                \right )_{1 \leq i,  j \leq r}.  \medskip 
\end{array} 
\end{equation*}  
Here we used the  following identity:  
\begin{equation*} 
   \binom{1-i}{k}  =  \binom{-(i-1)}{k}   =  (-1)^{k} \binom{i + k-2}{k}, 
\end{equation*} 
for   integers $i \geq 1$, $k \geq 0$.
This is the dual  Grothendieck polynomial  
$g_{\lambda}(\y_{n})$  introduced in (\ref{eqn:DefinitionDSGP}) with $\beta = -1$.  
\end{proof}

\subsection{Conjecture on a generating function for $gq_{\nu} (\y_{n})$}    \label{subsec:ConjectureGFgq}

 In \cite[\S 3.4]{Ikeda-Naruse2013},  Ikeda--Naruse introduced the $K$-theoretic $P$- and $Q$-functions $GP_{\nu}(\bm{x})$ and 
$GQ_{\nu}(\bm{x})$ in countably many variables $\bm{x} = (x_{1}, x_{2}, \ldots)$.  
Let $G\Gamma' (\bm{x})$ denote the ring of symmetric functions satisfying the $K$-theoretic $Q$-cancellation property $($see 
\cite[Definition 1.1]{Ikeda-Naruse2013}$)$. Similarly,   let $G \Gamma (\bm{x})$ denote the subring of  $G \Gamma' (\bm{x})$ 
 consisting of functions $f$ satisfying the condition$:$$f(t, x_{2}, \ldots )  - f(0, x_{2}, \ldots)$ is divisible by 
$t \oplus t$.\footnote{
We slightly changed the notation  from that used in \cite{Ikeda-Naruse2013}. 
In that paper,   $G \Gamma' (\bm{x})$ and $G \Gamma (\bm{x})$ are written as $G \Gamma$ and 
$G \Gamma_{+}$ respectively.
}     Then, they showed that  $GP_{\nu}(\bm{x})$'s   and $GQ_{\nu} (\bm{x})$'s    
$(\nu$ strict$)$ form a formal $\Z [\beta]$-basis 
of $G \Gamma' (\bm{x})$ and $G \Gamma (\bm{x})$ respectively.  
Using this  ``basis theorem'' and   the following  ``Cauchy kernel'' 
\begin{equation*} 
    \Delta (\bm{x}; \bm{y}) = \prod_{i=1}^{\infty} \prod_{j=1}^{\infty} \dfrac{1 - \overline{x}_{i} y_{j}} {1 - x_{i} y_{j}}, 
\end{equation*} 
where $\bm{y} = (y_{1}, y_{2}, \ldots)$ is another set of independent variables, 
we can define the {\it dual $K$-theoretic $P$- and $Q$-functions}, denoted by 
 $gp_{\nu}(\bm{y})$ and   $gq_{\nu}(\bm{y})$,  as follows (see also Nakagawa--Naruse \cite[Definition 5.3, Remark 5.4]{Nakagawa-Naruse2016}):     
\begin{defn} [Dual $K$-theoretic Schur $P$- and $Q$-functions]   \label{df:Definitiongpgq}
Let $\mathcal{S P}$ denote the set of all strict partitions.  We define 
$gp_{\nu}(\bm{y})$ and $gq_{\nu}(\bm{y})$ for a strict partition $\nu \in \mathcal{S P}$ by the 
following identities$:$  
\begin{equation}  \label{enq:Definitiongpgq} 
   \Delta (\bm{x}; \bm{y}) 
   = \displaystyle{\prod_{i=1}^{\infty}}   \prod_{j=1}^{\infty} \dfrac{1 - \overline{x}_{i} y_{j}} {1 - x_{i} y_{j}}
     =  \displaystyle{\sum_{\nu \in \mathcal{S P}}}    GP_{\nu}(\bm{x})  gq_{\nu} (\bm{y})    
     =   \displaystyle{\sum_{\nu \in \mathcal{S P}}}    GQ_{\nu}(\bm{x})  gp_{\nu} (\bm{y}).     
\end{equation} 
\end{defn}  
 One can check that $gp_{\nu}(\bm{y})$ and $gq_{\nu}(\bm{y})$ are actually symmetric functions, i.e., 
they are elements of $\Lambda (\bm{y}) \otimes \Z[\beta]$, where $\Lambda (\bm{y})$ is the ring of 
symmetric functions in the variables $\bm{y} = (y_{1}, y_{2}, \ldots)$ over $\Z$.  
For each  positive integer $n$,  one can define a surjective ring homomorphism 
$\rho^{(n)}: \Lambda (\bm{y}) \twoheadlongrightarrow \Lambda (\bm{y}_{n})$  
by  putting $y_{n + 1} = y_{n + 2} = \cdots = 0$.  
Here $\Lambda (\bm{y}_{n})  = \Z[y_{1}, \ldots, y_{n}]^{S_{n}}$ is the ring of symmetric polynomials 
in $\bm{y}_{n}  = (y_{1}, \ldots, y_{n})$ under the usual action by the symmetric group $S_{n}$.  
We also denote by $\rho^{(n)}$ its extension over $\Z[\beta]$.  
Then we define the {\it dual $K$-theoretic Schur $P$- and $Q$-polynomials}, 
denoted by $gp_{\nu}(\bm{y}_{n})$ and $gq_{\nu}(\bm{y}_{n})$ for a strict partition $\nu$
of length $\leq n$,  by $gp_{\nu}(\bm{y}_{n}) = \rho^{(n)} (gp_{\nu}(\bm{y}))$ 
and $gq_{\nu}(\bm{y}_{n}) = \rho^{(n)}(gq_{\nu}(\bm{y}))$ 
respectively.  

Next we set 
\begin{equation*} 
\begin{array}{llll} 
    gq (z) =  \displaystyle{\prod_{j=1}^{n}} 
           \dfrac{1 - y_{j} \overline{z}}  {1 - y_{j} z}  =  \sum_{k=0}^{\infty}  gq_{k}(\bm{y}_{n})  z^{k}, \medskip \\
    gq (\bm{z}_{r})  = gq(z_{1}, \ldots, z_{r}) 
 :=  \displaystyle{\prod_{i=1}^{r}}   gq (z_{i})  \prod_{1 \leq i <j \leq r}  \dfrac{z_{i} \ominus z_{j}} {z_{i} \oplus z_{j}}. 
\medskip 
\end{array} 
\end{equation*} 
Then we make the following conjectures: 
\begin{conj} [Generating function for $gq_{\nu}(\bm{y}_{n})$]  \label{enq:GFDKQP}  
For  a strict partition  $\nu = (\nu_{1}, \ldots, \nu_{r})$ 
of length $\ell (\nu)  = r \leq n$,  the dual $K$-theoretic $Q$-polynomial 
$gq_{\nu}(\bm{y}_{n})$  is the coefficient of $\bm{z}_{r}  = z_{1}^{\nu_{1}} z_{2}^{\nu_{2}} \cdots z_{r}^{\nu_{r}}$ 
in $gq(z_{1}, \ldots, z_{r})$.   Thus 
\begin{equation*} 
     gq_{\nu}(\bm{y}_{n})   =  [\bm{z}^{\nu}]   (gq (\bm{z}_{r})).      
\end{equation*} 
\end{conj}  
We have checked that the above conjecture holds for $r \leq 2$.  
As a corollary to the above conjecture, we immediately obtain the following formula: 
\begin{cor}[Pfaffian formula for $gq_{\nu}(\bm{y}_{n})$]  
For a strict partition $\nu$  of  length $2m$,   we have 
\begin{equation*}  
  gq_{\nu} (\bm{y}_{n})  =   \mathrm{Pf}  \left ( \sum_{k=0}^{i-1} \sum_{l=0}^{j}  \beta^{k + l}  
          \binom{i-1}{k} \binom{j}{l}  gq_{(\nu_{i} - k,  \nu_{j} - l)}  (\bm{y}_{n}) \right )_{1 \leq i < j \leq 2m}.    
\end{equation*} 
\end{cor}



\begin{thebibliography}{99}


\bibitem{Bremke-Malle1997} 
K. Bremke and G. Malle, 
{\it Reduced words and a length function for $G(e, 1, n)$}, 
Indag. Math. (N. S.) \textbf{8} (1997),  no. 4,  453--469. 



\bibitem{Bremke-Malle1998}
K. Bremke and G. Malle, 
{\it Root systems and length functions}, 
Geom.  Dedicata \textbf{72} (1998),  no. 1,  83--97. 




\bibitem{Brion1996} 
M. Brion, 
{\it The push-forward and Todd class of flag bundles}, 
Parameter Spaces (P. Pragacz, ed.), {\bf 36},  Banach Center Publications,  1996, 
45--50.  







\bibitem{Frame-Robinson-Thrall1954}
J. S. Frame,  G. de B. Robinson,  and R. W.  Thrall, 
{\it The hook graphs of the symmetric group}, 
Can. J. Math.  {\bf 6} (1954),   316--325.  





\bibitem{Fulton1997}   W. Fulton, 
{\it Young Tableaux}, 
London Mathematical Society Student Texts {\bf 35}, Cambridge University Press, 1997.  





\bibitem{Fulton-Pragacz1998} 
 W. Fulton and P. Pragacz, 
{\it  Schubert varieties and degeneracy loci}, 
 Lecture Notes in Mathematics  \textbf{1689}, Springer-Verlag, Berlin, 1998. 



\bibitem{Hiller1982}
H. Hiller, 
{\it  Combinatorics and intersection of Schubert varieties}, 
Comment. Math.  Helv.  \textbf{57}  (1982),  no. 1, 41--59.     

 





\bibitem{Ikeda2007}  T. Ikeda, 
{\it Schubert classes in the equivariant cohomology of the Lagrangian 
Grassmannian},  
Adv.  Math.  {\bf  215} (2007), no. 1, 1--23.   






\bibitem{IMN2011} 
T. Ikeda, L. C. Mihalcea, and H. Naruse, 
{\it Double Schubert polynomials for the classical groups}, 
{Adv.  Math.}   \textbf{226} (2011),   840--866.




\bibitem{Ikeda-Naruse2009}  
T. Ikeda and H. Naruse, 
{\it Excited Young diagrams and equivariant Schubert calculus}, 
{Trans. Amer. Math. Soc.}   {\bf 361} (2009),  5193--5221.  



\bibitem{Ikeda-Naruse2013}
T. Ikeda and H.  Naruse,
{\it $K$-theoretic analogues of factorial Schur $P$- and $Q$-functions},
{Adv.   Math.}  {\bf 243} (2013),  22--66. 




\bibitem{Ivanov2004}  
V. N. Ivanov, 
{\it Interpolation analogs   of Schur $Q$-functions}, 
{Zapiski  Nauchnykh Seminarov POMI}  {\bf 307} (2004), 99--119; 
{J. of Math. Sci.}  {\bf 131} (2)  (2005),  5495--5507. 





\bibitem{Jozefiak1991}  
T. J\'{o}zefiak, 
{\it Schur $Q$-functions and cohomology of isotropic Grasmannians}, 
Math. Proc. Camb. Phil. Soc. {\bf 109} (1991),  471--478. 






\bibitem{Knutson-Tao2003}   
A. Knutson and T. Tao, 
{\it Puzzles and $($equivariant$)$ cohomology of Grassmannians}, 
Duke Math. J. {\bf 119} (2003),  221--260. 



\bibitem{Lascoux-Naruse2014} A. Lascoux and H. Naruse, 
{\it  Finite sum Cauchy identity for dual Grothendieck polynomials}, 
Proc. Japan Acad.,  {\bf 90}, Ser. A (2014),  87--91.    






\bibitem{Macdonald1995}  
I. G. Macdonald, 
{\it Symmetric functions and Hall polynomials}, 
2nd edition, Oxford  Univ. Press, Oxford, 1995.





\bibitem{McDaniel2016} 
C. McDaniel, 
{\it A GKM description of the equivariant coinvariant ring of a pseudo-reflection group}, 
arXiv:1609.00849.  



\bibitem{Molev-Sagan1999}  
A. I.  Molev and B. E. Sagan, 
{\it A Littlewood-Richardson rule for factorial Schur functions}, 
{Trans. Amer. Math. Soc.}  {\bf 351} (1999),   4429--4443.  


\bibitem{Nakada2008} 
K. Nakada, 
{\it Colored hook formula for a generalized Young diagram}, 
Osaka J. Math.  {\bf 45} (2008),  no. 4,  1085--1120.  




\bibitem{Nakagawa-Naruse2016}
M. Nakagawa and H. Naruse, 
{\it Generalized $($co$)$ homology of the loop spaces of classical groups 
and the universal factorial Schur $P$-
and $Q$-functions}, 
{\it Schubert Calculus--Osaka} 2012, 
{\it Adv.  Stud.  Pure Math.}  {\bf 71},  {\it Math. Soc. Japan}, {\it Tokyo}, 2016, pp.337--417.  
 


\bibitem{Nakagawa-Naruse2018} 
M. Nakagawa and H. Naruse, 
{\it Universal Gysin formulas for the universal Hall--Littlewood functions}, 
Contemp. Math. {\bf 708}  (2018),   201--244. 


\bibitem{Nakagawa-Naruse2018(arXiv)} 
M. Nakagawa and H. Naruse, 
{\it Universal Factorial Schur $P$, $Q$-functions and their duals}, 
arXiv:1812.03328.  


\bibitem{Nakagawa-Naruse2019(arXiv)} 
M. Nakagawa and H. Naruse, 
{\it  Darondeau--Pragacz formulas  in complex cobordism}, 
Math. Ann. {\bf 381} (2021), no. 1-2, 335--361.  


 

\bibitem{Nakagawa-Naruse2022}
M. Nakagawa and H. Naruse, 
{\it Equivariant Schubert calculus for unitary reflection groups}, 
in preparation.   




\bibitem{Naruse2017} 
H. Naruse, 
{\it Elementary proof and application of the generating function for generalized 
Hall--Littlewood functions}, 
J. Algebra {\bf 516} (2018),  197--209. 



\bibitem{Naruse-Okada2019} 
H. Naruse and S. Okada, 
{\it Skew hook formula for $d$-complete posets via 
equivariant $K$-theory}, 
Algebr. Comb. {\bf 2} (2019), no. 4, 541--571.   





\bibitem{Ortiz2015}  O. Ortiz, 
{\it GKM theory for $p$-compact groups}, 
J. Algebra {\bf 427} (2015),  426--454.    


\bibitem{Pragacz1991}    
P.  Pragacz, 
{\it Algebro-geometric applications of  Schur $S$- and 
$Q$-polynomials}, 
Topics in invariant theory (Paris, 1989/1990), 130--191, 
Lecture Notes in Mathematics   {\bf 1478}, Springer, Berlin, 1991. 




\bibitem{Pragacz2015}  P. Pragacz, 
{\it A Gysin formula for Hall--Littlewood polynomials}, 
{Proc. Amer. Math. Soc.}  {\bf 143} (2015) (11),  4705--4711.    
   



\bibitem{Quillen1969}  
D. Quillen, 
{\it On the formal group laws of unoriented and complex cobordism theory}, 
{Bull. Amer. Math. Soc.}  {\bf 75} (6) (1969), 1293--1298. 




\bibitem{Quillen1971} 
D. Quillen, 
{\it Elementary proofs of some results of cobordism theory using Steenrod operations}, 
{Adv. in  Math.}  {\bf 7} (1971),   29--56.  







\bibitem{Totaro2003} 
B. Totaro, 
{\it Towards a Schubert calculus for complex reflection groups}, 
{Math. Proc. Camb. Phil. Soc.}  \textbf{134} (2003),  83--93.



\end{thebibliography}
\end{document}